\documentclass{article}

\usepackage[twoside,
paperwidth=210mm,
paperheight=297mm,
textheight=622pt,
textwidth=468pt,
centering,
headheight=50pt,
headsep=12pt,
footskip=18pt,
footnotesep=24pt plus 2pt minus 12pt,
columnsep=2pc]{geometry}

\usepackage[auth-sc]{authblk}

\usepackage{mathtools}

\allowdisplaybreaks

\usepackage{amssymb}
\usepackage[mathscr]{eucal}

\usepackage{mismath}

\usepackage{hyperref}

\usepackage{bm} 

\usepackage{graphicx}
\usepackage{pgf}
\usepackage{pgfplots}
\pgfplotsset{compat=1.18}

\graphicspath{{figs/}}

\usepackage{tikz}
\usetikzlibrary{math}
\usetikzlibrary{calc}

\usepackage{subcaption}

\usepackage{multirow}

\usepackage{makecell}
\setcellgapes{2pt}

\usepackage{booktabs}

\usepackage{siunitx}
\usepackage{algorithm}
\usepackage{algpseudocode}

\usepackage[skins]{tcolorbox}
\newtcolorbox{stepbox}[2][]{%
  enhanced,
  attach boxed title to top center={yshift=-3mm,yshifttext=-1mm},
  colframe=blue!75!black,
  colbacktitle=red!80!black,
  fonttitle=\bfseries,
  title=#2,#1
}

\newcommand{\Real}{\mathbb{R}}

\DeclareMathOperator{\Span}{span}

\DeclarePairedDelimiter{\RoundBrackets}{(}{)}

\usepackage{amsthm}

\usepackage[capitalise]{cleveref}

\newtheorem{theorem}{Theorem}[section]
\newtheorem{lemma}[theorem]{Lemma}

\theoremstyle{definition}

\newtheorem{assumption}{Assumption}

\crefname{assumption}{assumption}{assumptions}
\Crefname{assumption}{Assumption}{Assumptions}
\theoremstyle{remark}
\newtheorem*{remark}{Remark}
\usepackage[numbers,sort&compress]{natbib}
\crefname{equation}{}{}

\usepackage[utf8]{inputenc} 
\usepackage[T1]{fontenc}    
\usepackage{microtype}       
\usepackage{lineno}

\providecommand{\keywords}[1]{\textbf{\textit{Keywords---}} #1} 

\title{Efficient Multiscale Methods for Highly Heterogeneous Spatial Network Models}
\author[1]{Yingjie Zhou}
\author[1]{Xiang Zhong\thanks{Corresponding author.
(Email address: \href{mailto:xzhong@math.cuhk.edu.hk}{xzhong@math.cuhk.edu.hk})}}
\author[1]{Changqing Ye}
\author[1]{Eric T. Chung}
\affil[1]{Department of Mathematics, The Chinese University of Hong Kong, Shatin, Hong~Kong~SAR, China.}
\date{}
\begin{document}
\maketitle
\begin{abstract}
  Modeling complex spatial networks with multiscale heterogeneity poses significant mathematical and computational challenges. Lacking explicit PDE discretizations and facing excessive degrees of freedom, conventional methods often become computationally prohibitive. To address these challenges, we propose an efficient multiscale modeling for highly heterogeneous spatial networks. We construct multiscale basis functions tailored to spatial network models with heterogeneous edge weights and node degrees. A key novelty is that the proposed method doesn't introduce geometric parameters (such as Dirichlet nodes, distances, or mesh sizes), thereby preserving its purely algebraic nature and ensuring broad applicability. By incorporating a subgraph-wise estimate, we define a Poincaré constant $C_{\mathrm{po}}$ that renders the method independent of the underlying graph geometry. Then through an appropriate choice of the number of graph oversampling layers, we establish an $O(C_{\mathrm{po}})$ convergence independent of the local heterogeneity contrast. Notably, our scheme operates entirely within an algebraic framework, eliminating the need for Dirichlet nodes and positive-definiteness on specific matrices arising in the model. This flexibility enables the simulation of a wider range of physical models and accommodates various boundary conditions. Rigorous theoretical analyses are provided under suitable assumptions, and extensive numerical experiments validate the effectiveness of the proposed approach.

\end{abstract}

\keywords{spatial network models, highly heterogeneity, multiscale method} 

\section{Introduction}






Spatial network models are essential for analyzing complex physical systems with intricate connectivity and multiscale heterogeneity, with broad applications across science and engineering. In fiber-based materials such as paper, nonwoven fabrics, and composites, these models capture stochastic fiber geometry to predict thermal, mechanical, and transport properties \cite{KU2012,IH2009}. For porous media flow, network representations link pore-scale physics to macroscopic transport in reservoirs and fuel cells \cite{W2022,C2012}. In biomedical applications, vascular network models are key to understanding hemodynamics, oxygen and nutrient delivery \cite{A2010}. Wave propagation in geophysical and acoustic systems also relies on efficient modeling in heterogeneous networks \cite{GLM2023}. Despite their utility as reduced models, spatial network models can still pose significant computational challenges when solving large-scale systems derived from governing equations on complex network topologies.

Spatial network models with a large number of degrees of freedom (DOFs) often lead to algebraic systems that are difficult to solve directly. Black-box preconditioners, such as Algebraic Multigrid (AMG) methods, are widely used to accelerate this process by exploiting the matrix graph structure \cite{M1987,XZ2017}. However, AMG methods are iterative in nature and lack interpretability, as they do not provide an explicitly upscaled model. In the context of PDE-based models, a common strategy for dimension reduction is to integrate coefficient heterogeneity into basis function construction. This principle underpins numerous multiscale methods, including multiscale finite element methods (MsFEMs) \cite{YTY2009,HTY1997}, the localized orthogonal decomposition method (LOD) \cite{AMDP2014}, variational multiscale finite element methods \cite{HFJL1998}, numerical upscaling \cite{PDR2016}, heterogeneous multiscale methods (HMM) \cite{EW2003,EWBE2003,BEYH2005} and density-equalizing method \cite{LyuChoiLui2024a, LyuLuiChoi2024b}.

By analogy, the heterogeneity in network models—encoded in edge weights and node degrees—plays a role similar to coefficient variations in PDEs. This suggests that incorporating ideas from multiscale computational methods could be a promising avenue for deriving an upscaled network model. A key distinction, however, is that network models lack the explicit coarse grids used in PDE discretizations. Instead, general graph partitions must serve as a substitute for two-scale nested grids. In special cases where the network exhibits periodicity, homogenization theory becomes relevant \cite{BLP2011,JKO2012}. As the periodic cell size approaches zero, the network model is expected to converge to a continuum limit. Notably, mechanics research indicates that this limit may involve nonlocal effects or high gradients, deviating from classical homogenization theories \cite{B2017,DK2017,KM2004}.

Recent advancements in numerical methods for spatial network models have predominantly utilized LOD and subspace decomposition techniques to construct efficient coarse-scale approximations~\cite{K2020,E2024,ILW2010}, with recent extensions including algebraic multiscale methods~\cite{HMM2023}, super-localized approaches~\cite{HM2024}, and preconditioned iterative solvers~\cite{GHM2024}.  While these methods have demonstrated robustness under specific network assumptions (e.g., homogeneity and connectivity)~\cite{E2024,HM2024}, their theoretical guarantees often rely on the construction of stable quasi-interpolation operators and exponential decay estimates of corrector functions. Ensuring these properties on complex, unstructured networks with extreme heterogeneities may require restrictive structural assumptions or large oversampling parameters. In contrast, the Constrained Energy Minimization Generalized Multiscale Finite Element Method (CEM-GMsFEM) \cite{Eric2018,EricYETY2023,chung2025locking,ChungKimZhong2026,JinLiuZhongChung2025,Eric_mix_2018,YeEric2023,ChungHuPun2021,FEW2020,EricFu2023,EricSM2020} offers a robust alternative by constructing multiscale basis functions through energy minimization subject to coarse-grid constraints. This approach inherently handles high-contrast coefficients without relying on specific interpolation stability, provides systematic control over the energy norm error, and allows for flexibility in incorporating local spectral information, thereby providing a robust methodology inherent in scenarios with a high number of DOFs for spatial network models. 

In this work, we propose an efficient Constrained Energy Minimization Generalized Multiscale Finite Element Method for spatial network models. We consider an abstract model defined on a spatial network $\mathcal{G}=(\mathcal{N},\mathcal{E})$, where $\mathcal{N}$ and $\mathcal{E}$ denote the sets of nodes and edges, respectively. We consider two symmetric non-negative definite matrices: $L$, which encodes a weighted graph Laplacian exhibiting high contrast, and $M$, which captures the mass properties. Then we aim to solve an equation of the following form: find a function $u\in V$ such that
\begin{equation*}
 ((L+M)u,v)=(f,v) 
\end{equation*}    
for all $v\in V$, given the source term $f\in V$ and $(\cdot,\cdot)$ denoting the Euclidean scalar product taken over the nodes of the network. In the proposed ACEM-GMsFEM, we construct multiscale basis functions tailored for spatial network models with high heterogeneity in two stages.
In the first stage, we construct an auxiliary multiscale space for each subgraph by solving local spectral problems. The connections between adjacent subgraphs are carefully addressed. By incorporating a subgraph-wise Poincaré inequality, we introduce an importance constant $C_{\mathrm{po}}$ that makes the method purely algebraic and independent of graph geometry.
In the second stage, we utilize this auxiliary multiscale space to construct a multiscale space.
By appropriately selecting the number of graph oversampling layers, we achieve an $O(C_{\mathrm{po}})$ convergence, independent of the local contrast.  We provide detailed theoretical analyses and numerical experiments to validate our approach’s efficacy across challenging scenarios.  

Different from LOD-based methods such as \cite{HMM2023,HM2024,GHM2024}, our method offers several key advantages. First, it operates in a purely algebraic framework without requiring Dirichlet nodes or imposing positive-definiteness on $L$ and $M$. Second, this flexibility enables the simulation of a broader class of physical models, accommodating various boundary conditions (Dirichlet, Neumann, or Robin) for both steady and dynamic problems such as \cite{zhong2023spectral, ZhongQiu2024}. Finally, the method's analysis is free from geometric parameters (e.g., distance or mesh sizes), which underscores its purely algebraic nature (using a Poincaré constant $C_{\mathrm{po}}$) and ensures its wide applicability.

This paper is organized as follows. In Section \ref{sec:Spatial network model}, we present the spatial network model, introduce some
notation and assumption for the oprerators in the spatial network model. Coarse partition and multiscale basis functions tailored for spatial network models are constructed in Section \ref{multiscale construction}. Detailed analysis of first-order convergence with respect to the Poincaré constant $C_{\mathrm{po}}$ is presented in Section \ref{analysis}. This section further establishes that the number of small eigenvalues corresponds to the count of disconnected high-contrast inclusions or channels, and that the smallest eigenvalue remains $\mathcal{O}(1)$, independent of the contrast. In Section \ref{application}, we provide a straightforward application of our method to elliptic PDEs, yielding an approximation that retains the similar properties in the continuous setting. To validate the performance of the proposed method, we report numerical experiments in Section \ref{numerical results}. Finally, a summary of our conclusions is given in Section \ref{conclusions}.

\section{Spatial network model}\label{sec:Spatial network model}
In this section, we describe the spatial network model, introduce the notation, and state the necessary assumptions.

The spatial network we are considering is an undirected connected graph, represented by $\mathcal{G}=\{\mathcal{N},\mathcal{E}\}$, where $\mathcal{N}$ denotes the node set and $\mathcal{E}$ represents the edge set. 
For $x,y\in\mathcal{N}$, we write $x\sim y$ if $x$ and $y$ are connected by an edge. Therefore, $\mathcal{E}=\{\{x,y\}\colon x\sim y, x,y\in\mathcal{N}\}$. 
A node $x$ is a neighbor of node $y$ if there is an edge $\{x,y\}\in\mathcal{E}$. 

We define the vector space $V$ as the space of all real-valued functions defined on the nodes $\mathcal{N}$.
Since $V$ is a finite dimensional space, and it is an isomorphism to $\Real^{\# \mathcal{N}}$, while this notation may be reminiscent of PDE discretizations. 

For any subset of nodes $\mathcal{R}\subset\mathcal{N}$, we define $V_{\mathcal{R}}$ as the space of real-valued functions defined on $\mathcal{R}$.

Through the article, for a function $f$ defined on the graph, we denote $f(x)$ its value on the node $x \in \mathcal{N}$. 
We define a diagonal weighted mass operator $M\colon V\to V$ by the following sum of nodal contributions:
\begin{equation}
\label{Mass}
(Mv,w)\coloneqq\sum_{x\in\mathcal{N}}M_xv(x)w(x),
\end{equation}
where $0\leq M_x\leq \overline{M}<\infty$ and $(\cdot,\cdot)$ denotes the Euclidean inner product of vectors of nodal values of two elements in $V$. Note that $\abs{\cdot}_M^2=(M\cdot,\cdot)$ is a semi norm. For a subset $\mathcal{R}\subset\mathcal{N}$, we define localized versions of the mass operators by $(M_{\mathcal{R}}v,w)\coloneqq\sum_{x\in \mathcal{R}}M_xv(x)w(x)$. We also introduce an notation for the summation over edges as
\[
\sum_{\mathcal{R}, \sim} \coloneqq \frac{1}{2}\sum_{x\in \mathcal{R}} \sum_{y\in \mathcal{R},y \sim x}.
\]
Then
 The local weighted stiffness operator $L\colon V\to V$ is given by 
\begin{equation}
\label{def_L}
(Lv,w)\coloneqq \sum\limits_{\mathcal{N},\sim}L_{xy}(v(x)-v(y))\left(w\left(x\right)-w\left(y\right)\right),
\end{equation}
where $0<\underline{L}\leq L_{xy}\leq \overline{L}<\infty$. 
For all $x \in \mathcal{N}$, we denote $$\tilde{L}_x \coloneqq \sum_{y\in \mathcal{N},y\sim x}L_{xy}/2$$ as lumping all weights over nodes adjacent to $x$.
For subsets $\mathcal{R}\subset\mathcal{N}$, we define localized versions $L_{\mathcal{R}}$ of the operator $L$ as  $(L_{\mathcal{R}}v,w)\coloneqq\sum\limits_{x\in \mathcal{R}}\frac{1}{2}\sum\limits_{y\sim x}L_{xy}(v(x)-v(y))(w(x)-w(y))$. 
Note that $(L_{\mathcal{R}}v)(x)$ is non-zero for node $x$ outside of $\mathcal{R}$ that are adjacent to nodes in $\mathcal{R}$. 
Clearly, if $\mathcal{N}=\mathcal{R}_1 \cup \mathcal{R}_2$ with $\mathcal{R}_1 \cap \mathcal{R}_2=\varnothing$, we should have $(Lv,v)=(L_{\mathcal{R}_1}v,v)+(L_{\mathcal{R}_2}v,v)$.
Note $\abs{\cdot}_L^2=(L\cdot,\cdot)$ defines a semi-norm on $V$. Localized version of this norm for $\mathcal{R}\subset\mathcal{N}$ are denoted by $\abs{\cdot}_{L,\mathcal{R}}^2\coloneqq(L_{\mathcal{R}}\cdot,\cdot)$.

Given a source term $f\in V$, our model problem seeks a function $u\in V$ such that
\begin{equation}
\label{model problem}
    a(u,v):= ((L+M)u,v)=(f,v) \quad \forall v\in V.
\end{equation}
To simplify the discussion, throughout the paper, we give the following assumption.
\begin{assumption}
\label{assumption on spd}
The operator $L+M\colon V\to V$ is symmetric positive definite, i.e., $((L+M)v,v)>0$ for all $v\in V\setminus\{\mathbf{0}\}$.
\end{assumption}
Since $\mathcal{G}$ is connected, $\ker(L)$ consists only of globally constant functions. Assumption \ref{assumption on spd} is satisfied if either of the following conditions holds:
\begin{enumerate}
    \item[(i)] there exists at least one node $x\in\mathcal{N}$ with $M_x>0$, so that the bilinear form $(M\cdot,\cdot)$ defines an inner product on $V$ with induced norm $\abs{\cdot}_M^2=(M\cdot,\cdot)$; or
    \item[(ii)] $v(x)=0$ is enforced for at least one node $x\in\mathcal{N}$, in which case the only constant function in $V$ is the zero function and $L$ is already symmetric positive definite on $V$.
\end{enumerate}
Under Assumption \ref{assumption on spd}, $((L+M)\cdot,\cdot)$ is an inner product on $V$. 
We write the induced norm as $\norm{\cdot}_a^2=\abs{\cdot}_M^2+\abs{\cdot}_L^2$.
For subsets $\mathcal{R}\subset\mathcal{N}$, we define the localized energy norm by
$$\norm{v}_{a(\mathcal{R})}^2\coloneqq(L_{\mathcal{R}}v,v)+(M_{\mathcal{R}}v,v)=\sum_{x\in\mathcal{R}}\frac{1}{2}\sum_{y\sim x}L_{xy}(v(x)-v(y))^2+\sum_{x\in\mathcal{R}}M_xv(x)^2.$$
For $v\in V$, $\norm{\cdot}_{a(\mathcal{R})}$ is a semi-norm that localizes the global energy to $\mathcal{R}$. When restricted to a subset $\mathcal{R}\subsetneq\mathcal{N}$ with zero extension outside $\mathcal{R}$, $\norm{\cdot}_{a(\mathcal{R})}$ becomes a norm.

The unique solvability of  (\ref{model problem}) can be concluded from that $((L+M)\cdot,\cdot)$ is an inner product on $V$. Noting that $(f,\cdot)\in V'$, where $V'$ denotes the dual space of $V$, the unique solvability follows from the Riesz representation theorem.
We will discuss the construction of multiscale basis functions in the next section. We consider $V_\mathup{ms}^m$ to be the space spanned by all multiscale basis functions. Then the multiscale solution $u_\mathup{ms}$ is denoted as the solution of the following problem: find $u_\mathup{ms}\in V_\mathup{ms}^m$ such that
\begin{equation}
\label{multiscale problem}
a(u_\mathup{ms},v)=(f,v) \quad \forall v\in V_\mathup{ms}^m.
\end{equation}

The degree $d_x$ of a node $x$ is the number of its neighbors. 
For $\mathcal{R}\subset\mathcal{N}$, the degree of the node $x\in \mathcal{R}$ is denoted by $d_{\mathcal{R}}(x)$, which counts the number of nodes in $\mathcal{R}$ that are adjacent to $x$. 
We define the volume of $\mathcal{R}\subset\mathcal{N}$ by ${\rm vol}(\mathcal{R})\coloneqq\sum_{x\in\mathcal{R}}d_x$; Besides, we define the volume of $\mathcal{R}_1\subset \mathcal{R}_2\subset \mathcal{N}$ by ${\rm vol}_{\mathcal{R}_2}(\mathcal{R}_1)\coloneqq\sum_{x\in \mathcal{R}_1}d_{\mathcal{R}_2}(x)$.

\section{Coarse partition and construction of the multiscale basis functions} \label{multiscale construction}
In this section, we first introduce some basic notation for spatial network models. We will present the construction of the auxiliary multiscale basis functions. 
\subsection{Partitions}\label{partition}
Cutting a graph into smaller pieces is a fundamental algorithmic operation \cite{AHIP2016}. In graph partitioning, the key task is to divide the node set $\mathcal{N}$ into non-overlapping parts. That is, $\mathcal{N}=\cup_{i=1}^N\mathcal{N}_i$ with  $\mathcal{N}_i\cap\mathcal{N}_j=\varnothing$.  Similar to the notation of finite element meshes, we write $\mathcal{T}\coloneqq\{\mathcal{N}_i\colon 1\leq i\leq N\}$. We call $\mathcal{N}_i$ subgraph. Assume that for any partition, the corresponding subgraphs are connected. We denote the edge set $\mathcal{E}_{i,j}\coloneqq\{x\sim y\colon x\in\mathcal{N}_i,y\in\mathcal{N}_j\}$. For any $\mathcal{N}_i$, define the oversampling subgraph $\mathcal{N}_i^k$ recursively by
$$\mathcal{N}_i^k\coloneqq\mathcal{N}_i^{k-1}\cup\{\mathcal{N}_j\colon\exists x\in\mathcal{N}_j, y\in\mathcal{N}_i^{k-1}, {\rm  such \hspace{0.5em}that\hspace{0.5em}} x\sim y\}$$
with $\mathcal{N}_i^0\coloneqq\mathcal{N}_i$.

We impose a regularity assumption on the coarse partition $\mathcal{T}$, which quantifies the shape and local complexity of the subgraphs. Although the following constants always exist for any finite graph, their magnitude directly influences the convergence behavior of the proposed method.
\begin{assumption}
\label{assumption on partitions}
\begin{enumerate}
    \item[(1)] There exists a function $F\colon\mathbb{N}\to\mathbb{N}$ such that for all $1\leq i\leq N$ and $l>0$,
$$\#\{\mathcal{N}_j\colon 1\leq j\leq N,\ \mathcal{N}_j\subset\mathcal{N}_i^l\}\leq F(l).$$
    \item[(2)] There exists a constant $\tilde{C}>0$ such that for any $x\in\mathcal{N}$,
$$d_x\leq\tilde{C}.$$
\end{enumerate}
\end{assumption}

\subsection{The construction of the multiscale basis functions}
\label{construction of the multiscale basis functions}
In this section, we construct the multiscale basis functions based on the partitions proposed in Section \ref{partition}.

Let $V(\mathcal{N}_i)$ be the the space of real-valued functions defined on $\mathcal{N}_i$. 
We need to find a real number $\lambda^i_j$ and a function $\phi^i_j\in V(\mathcal{N}_i)$ such that
\begin{equation}
\label{spectral problem}
a_i(\phi^i_j,w)=\lambda^i_js_i(\phi^i_j,w)\quad \forall w\in V(\mathcal{N}_i),
\end{equation}
where
\begin{subequations}
	\label{def of a and s}
	\begin{align}
		a_i(v,w)&\coloneqq \sum_{\mathcal{\mathcal{N}}_i,\sim}L_{xy}(v(x)-v(y))(w(x)-w(y))+\sum_{x\in\mathcal{N}_i}M_xv(x) w(x), \label{def of a and s_a}\\
		s_i(v,w)&=\sum_{x\in\mathcal{N}_i}(\tilde{L}_{x}+M_x)(C_{\mathrm{po},i})^{-2}v(x) w(x).\label{def of a and s_b} 
	\end{align}
\end{subequations}
The scaling factor $C_{\mathrm{po}} > 0$ denotes the subgraph-wise Poincaré constant, which characterizes the intrinsic stability of the local neighborhood $\mathcal{N}_i$. Analytically, it is defined as the optimal embedding constant from the energy space into the weighted $L^2$ space. For notational simplicity, we omit the local index $i$ and define $C_{\mathrm{po}}$ for each $\mathcal{N}_i$ as
$
C_{\mathrm{po}}^{2} := \sup_{v \in V(\mathcal{N}_i) \setminus \{0\}} \frac{\sum_{x \in \mathcal{N}_i} (\tilde{L}_x + M_x) v(x)^2}{a_i(v, v)}.
$
The magnitude of $C_{\mathrm{po}}$ is primarily governed by the spectral gap of the local graph Laplacian and is highly sensitive to topological features, including connectivity, degree distributions, and potential bottlenecks. For simplicity, we abuse notation as $C_{\mathrm{po}} = \max_{1 \leq i \leq N} C_{\mathrm{po},i}$ in subsequent convergence analyses.

We assume the normalization that $s_i(\phi^{i}_k,\phi^{i}_l)=\delta_{kl}$. We remark that the choice of the bilinear  form (\ref{def of a and s_b}) is motivated by our convergence analysis. We let $\lambda^i_j$ the eigenvalues of (\ref{spectral problem}) arranged in ascending order. We also define $\abs{v}_{a_i}=a_i(v,v)^{\frac{1}{2}}$.
Note that by the definition of $a(\cdot,\cdot)$ and $a_i(\cdot,\cdot)$ for $1\leq i\leq N$, it is clearly that
\begin{equation}
\label{comparison of a}
\sum_{i=1}^N\abs{v}_{a_i}^2\leq\norm{v}_{a}^2
\end{equation}
for all $v\in V$.
We will use the first $l_i$ eigenfunctions to construct our local auxiliary multisacle space $V_\mathup{aux}^{i}$, where $V_\mathup{aux}^{i}={\Span}\{\phi^i_j \colon 1\leq j\leq l_i\}$. The global auxiliary multiscale space $V_\mathup{aux}$ is the sum of these local auxiliary multiscale space, namely $V_\mathup{aux}=\oplus_{i=1}^NV_\mathup{aux}^{i}$. This space is used to construct multiscale basis functions that are $\phi$-orthogonal to the auxiliary spaces defined above. The notation of $\phi$-orthogonality will be defined next.

For the local auxiliary multiscale space $V_\mathup{aux}^{i}$, the bilinear form $s_i$ in (\ref{def of a and s_b}) defines an inner product with norm $\norm{v}_{s(\mathcal{N}_i)}=\sqrt{s_i(v,v)}$. These local inner products and norms provide natural definitions of inner product and norm for $V$, which are defined by
$$s(v,w)=\sum_{i=1}^Ns_i(v,w),\quad \norm{v}_s=\sqrt{s(v,v)},\quad \forall v\in V_\mathup{aux}.$$
Let $\norm{\cdot}_s^2=s(\cdot,\cdot)$. Using the above inner product, we can define the notation of $\phi$-orthogonality in the space $V$. Given a function $\phi^i_j\in V_\mathup{aux}$, we say that a function $\psi\in V$ is $\phi^i_j$-orthogonal if
$$s(\psi,\phi^i_j)=1,\quad s(\psi,\phi^{i'}_{j'})=0,\hspace{0.5em} {\rm if}\hspace{0.5em} j'\neq j\hspace{0.5em}{\rm or}\hspace{0.5em}i'\neq i.$$
Now, we let $\pi_i\colon V(\mathcal{N}_i)\to V_\mathup{aux}^{i}$ be the projection with respect to the inner product $s_i(v,w)$. So, the operator $\pi_i$ is given by
$$\pi_i(u)=\sum_{j=1}^{l_i}\frac{s_i(u,\phi^i_j)}{s_i(\phi^i_j,\phi^i_j)}\phi^i_j,\quad \forall u\in V(\mathcal{N}_i).$$
We can easily derive the following basic estimates: for all $v\in V(\mathcal{N}_i)$,
\begin{equation}
\label{basic estimates for pi}
\norm{v-\pi_iv}_{s(\mathcal{N}_i)}^2\leq (\lambda^i_{l_i+1})^{-1}\norm{v}_{a_i}^2.
\end{equation}
\begin{equation}
\label{basic estimates for pi_2}
\norm{\pi_iv}_{s(\mathcal{N}_i)}^2=s_i(\pi_iv,v)\leq\norm{\pi_iv}_{s(\mathcal{N}_i)}\norm{v}_{s(\mathcal{N}_i)}\Rightarrow \norm{\pi_iv}_{s(\mathcal{N}_i)}\leq\norm{v}_{s(\mathcal{N}_i)}.
\end{equation}
In addition, we let $\pi\colon V\to V_\mathup{aux}$ be the projection with respect to the inner product $s(v,w)$. So, the operator $\pi$ is given by
$$\pi u=\sum_{i=1}^N\sum_{j=1}^{l_i}\frac{s_i(u,\phi^i_j)}{s_i(\phi^i_j,\phi^i_j)}\phi^i_j,\quad \forall u\in V.$$
Note that $\pi=\sum_{i=1}^N\pi_i$.

Next we consider the relaxed constraint energy minimizing generalized multiscale finite element method: find $\psi_{j,l}^i\in V(\mathcal{N}_i^l)$ such that
\begin{equation}
\label{relaxed cem}
\psi_{j,l}^i={\rm argmin}\{a(\psi,\psi)+s(\pi \psi-\phi^i_j,\pi \psi -\phi^i_j)\colon\psi\in V(\mathcal{N}_i^l)\},
\end{equation}
where $V(\mathcal{N}_i^l)$ is the the space of real-valued functions defined on $\mathcal{N}_i^l$. Our multiscale finite element space $V_\mathup{ms}^l$ is defined by
$$V_\mathup{ms}^l={\rm span}\{\psi_{j,l}^i\colon 1\leq j\leq l_i,\hspace{0.5em}1\leq i\leq N\}.$$
This minimization problem is equivalent to the following variational formulation
\begin{equation}
\label{variational_relaxed_cem}
a(\psi_{j,l}^i,v)+s(\pi \psi_{j,l}^i,\pi v )=s(\phi^i_j,\pi v),\hspace{0.5em}\forall v\in V(\mathcal{N}_i^l).
\end{equation}
The global multiscale basis function $\psi_{j}^{i}\in V$ is defined in a similar way, namely,
\begin{equation}
\label{global relaxed cem}
\psi_{j}^{i}={\rm argmin}\{a(\psi,\psi)+s(\pi(\psi)-\phi^i_j,\pi \psi-\phi^i_j)\colon \psi\in V\},
\end{equation}
which is equivalent to the following variational form
\begin{equation}
\label{global_variational_relaxed_cem}
a(\psi_{j}^{i},v)+s(\pi \psi_{j}^{i},\pi v)=s(\phi^i_j,\pi v),\quad \forall v\in V.
\end{equation}
Our multiscale finite element space $V^\mathup{glo}_\mathup{ms}$ is defined by
$$V^\mathup{glo}_\mathup{ms}={\rm span}\{\psi_{j}^{i}\colon 1\leq j\leq l_i,\hspace{0.5em}1\leq i\leq N\}.$$
This global multiscale finite element space $V^\mathup{glo}_\mathup{ms}$ satisfies a very important orthogonality property, which will be used in our convergence analysis. 

We remark that these basis functions are linearly independent. More precisely, we let $\sum_{i=1}^N\sum_{j=1}^{l_i}c_{ij}\psi_{j,l}^i=0$ and then substitute it into (\ref{variational_relaxed_cem}). 
We obtain that $0=s(\sum_{i=1}^N\sum_{j=1}^{l_i}c_{ij}\phi^i_j,\pi v)=\sum_{i=1}^N\sum_{j=1}^{l_i}c_{ij}s_i(\phi^i_j,\pi v)$ for all  $v\in V(\mathcal{N}_i^l)$.  Let the test function $v=\phi^i_j$ for all $1\leq i\leq N, 1\leq j\leq l_i$,  we clearly have $c_{ij}=0$,  for all $1\leq i\leq N, 1\leq j\leq l_i$. Thus, $\{\psi_{j,l}^i\colon 1\leq j\leq l_i,\hspace{0.5em}1\leq i\leq N\}$ are linearly independent.  
Similarly, $\{\psi_{j}^{i}\colon 1\leq j\leq l_i,\hspace{0.5em}1\leq i\leq N\}$ are also linearly independent.  

In particular, we define $\tilde{V}$ as the null space of the projection $\pi$, namely, $\tilde{V}=\{v\in V\colon \pi(v)=0\}$. Then for any $\psi^{i}_j\in V^\mathup{glo}_\mathup{ms}$, by (\ref{global_variational_relaxed_cem}), we have
$$a(\psi_{j}^{i},v)=0,\quad \forall v\in \tilde{V}.$$
That is, $\tilde{V}\subset (V^\mathup{glo}_\mathup{ms})^{\perp}$ and since  dim($V^\mathup{glo}_\mathup{ms}$)=dim($V_\mathup{aux}$), we have $\tilde{V}=(V^\mathup{glo}_\mathup{ms})^{\perp}$. Thus, we have $V=V^\mathup{glo}_\mathup{ms}\oplus\tilde{V}$.

\begin{lemma}[see \cite{YeEric2023,Eric2018}]
\label{a_orthogonality}
Let $v\in V^\mathup{glo}_\mathup{ms}$; then $a(v,v')=0$ for any $v'\in V$ with $\pi v'=0$ (i.e. for any $v'\in \tilde{V}$). Moreover, if there exists $v\in V$ such that $a(v,v')=0$ for any $v'\in  V^\mathup{glo}_\mathup{ms}$, then $\pi v=0$ (i.e. $v\in \tilde{V}$). 
\end{lemma}

\section{Analysis} \label{analysis}
This section establishes the convergence of the proposed method (Section \ref{Convergence analyses}) and characterizes the relationship between the number of small eigenvalues and that of disconnected high-contrast inclusions or channels (Section \ref{number of small eigenvalues}).

\subsection{Convergence analyses}\label{Convergence analyses}
In this section, we first prove the global convergence of the proposed method in Lemma \ref{global convergence}. To facilitate the analysis, we introduce abstract partition of unity functions in Assumption \ref{partition of unity} and define cutoff functions. Building on these, Lemma \ref{basis decay property} establishes a crucial decay property of our multiscale basis functions. Finally, by combining these results, Theorem \ref{convergence analysis} proves the local convergence of the method.

Before proving the convergence of the method, we need to define some notations. For a given subdomain $\mathcal{R}\subset\mathcal{N}$, we define the local $a$-norm and $s$-norm by
$$\norm{u}_{a(\mathcal{R})}^2=\sum_{x\in \mathcal{R}}\frac{1}{2}\sum_{y\sim x }L_{xy}(u(x)-u(y))^2+\sum_{x\in\mathcal{R}}M_xu(x)^2,$$
$$\norm{u}_{s(\mathcal{R})}^2=\sum_{x\in\mathcal{R}}(\tilde{L}_{x}+M_x)(C_{\mathrm{po}})^{-2}u(x)^2.$$
To prove the convergence of the proposed method, we will first show the convergence result of using the global multiscale basis functions. The approximation solution $u_\mathup{glo}\in V^\mathup{glo}_\mathup{ms}$ obtained in the global multiscale space $V^\mathup{glo}_\mathup{ms}$ is defined by
\begin{equation}
\label{u_glo}
a(u_\mathup{glo},v)=(f,v)\quad \forall v\in V^\mathup{glo}_\mathup{ms}.
\end{equation}
\begin{lemma}
\label{global convergence}
Let $u$ be the solution of (\ref{model problem}) and $u_\mathup{glo}$ be the solution of  (\ref{u_glo}). We have $u-u_\mathup{glo}\in\tilde{V}$ and
$$\norm{u-u_\mathup{glo}}_a\leq\frac{1}{\sqrt{\Lambda}}\norm{f}_{s^*},$$
where $$\Lambda=\min_{1\leq i\leq N}\lambda_{l_i+1}^{i} , \quad\norm{f}_{s^*}\coloneqq\sup_{v\in V, v\neq0}\frac{(f,v)}{\norm{v}_s}.$$
\end{lemma}
\begin{proof}
By the definition of $u$ and $u_\mathup{glo}$, we have
$$a(u-u_\mathup{glo},v)=0,\quad \forall v\in V^\mathup{glo}_\mathup{ms}.$$
Thus, we have $u-u_\mathup{glo}\in (V^\mathup{glo}_\mathup{ms})^{\perp}=\tilde{V}$. Using the orthogonality property and problem (\ref{model problem}), we obtain
\begin{equation*}
	\label{orthogonality}
	\begin{aligned}
		a(u-u_\mathup{glo},u-u_\mathup{glo})&=a(u,u-u_\mathup{glo})=(f,u-u_\mathup{glo})\\
  &\leq \norm{f}_{s^*}\norm{u-u_\mathup{glo}}_s.
	\end{aligned}
\end{equation*}
Since $u-u_\mathup{glo}\in \tilde{V}$, we have $\pi(u-u_\mathup{glo})=0$ by Lemma \ref{a_orthogonality}. 
By $\mathcal{N}_i\cap\mathcal{N}_j=\varnothing$ for $i\neq j$, we have $\pi_i(u-u_\mathup{glo})=0$ for all $i=1,2,...,N$. Therefore, we have
$$\norm{u-u_\mathup{glo}}_s^2=\sum_{i=1}^N\norm{u-u_\mathup{glo}}_{s(\mathcal{N}_i)}^2=\sum_{i=1}^N\norm{(I-\pi_i)(u-u_\mathup{glo})}_{s(\mathcal{N}_i)}^2.$$
By (\ref{basic estimates for pi}) and  (\ref{comparison of a}), we have
$$\sum_{i=1}^N\norm{(I-\pi_i)(u-u_\mathup{glo})}_{s(\mathcal{N}_i)}^2\leq \frac{1}{\Lambda}\sum_{i=1}^N\norm{u-u_\mathup{glo}}_{a_i}^2\leq\frac{1}{\Lambda}\norm{u-u_\mathup{glo}}_{a}^2.$$
Then we have $\norm{u-u_\mathup{glo}}_a^2=a(u-u_\mathup{glo},u-u_\mathup{glo})\leq \norm{f}_{s^*}\norm{u-u_\mathup{glo}}_{a}/{\sqrt{\Lambda}}$. The desired estimate is obtained.
\end{proof}

Unlike the standard finite element setting (where the mesh is triangular or quadrilateral), we can't use the standard Lagrange basis functions in the algebraic setting. So we introduce the abstract partition of unity functions as follows.
\begin{assumption}
\label{partition of unity}
Suppose there exists an overlapped partition
$\{U_i\}_{1\leq i\leq N}$ of the node set $\mathcal{N}$ with $U_i\supset\mathcal{N}_i$ for all $1\leq i\leq N$ and corresponding partition of unity functions $\{\eta_i\}_{1\leq i\leq N}$ with supp$(\eta_i)\subset U_i$ such that the following requirements are satisfied:
\begin{enumerate}
    \item[(1)] The functions $\{\eta_i\}_{1\leq i\leq N}$ are non-negative and satisfy $\sum_{i=1}^N\eta_i\equiv 1$ on $\mathcal{N}$.
    \item[(2)] There exists a constant $C_{\eta}>0$ independent of $C_{\mathrm{po}}$ such that
$$\max_{\{x,y\}\in\mathcal{E}}\abs{\eta_i(x)-\eta_i(y)}\leq C_{\eta}(C_{\mathrm{po}})^{-1},\quad\forall 1\leq i\leq N.$$
    \item[(3)] For all $1\leq i\leq N$,
$${\rm supp}(\eta_i)\subset U_i\subset \mathcal{N}_i^1.$$
\end{enumerate}
\end{assumption}

Before estimating the difference between the global and multiscale basis function, we need to construct the cutoff function with respect to oversampling subgraphs. For $M-m\geq 2$, we define $\chi_i^{M,m}\in{\rm span}\{\eta_i\}$ such that
\begin{subequations}
	\label{cutoff function}
	\begin{align}
\chi_i^{M,m}&\equiv1\quad {\rm in}\hspace{0.5em}\mathcal{N}_i^m,\label{cutoff function_a}\\
\chi_i^{M,m}&\equiv0\quad {\rm in}\hspace{0.5em}\mathcal{N}\backslash\mathcal{N}_i^M,\label{cutoff function_b}\\
0\leq\chi_i^{M,m}&\leq1\quad {\rm in}\hspace{0.5em}\mathcal{N}_i^M\backslash\mathcal{N}_i^m.\label{cutoff function_c}
\end{align}
\end{subequations}
Note that $\chi_i^{M,m}$ constructed as above exists. 
In fact, we can take $\chi_i^{M,m}=\sum_j\eta_j$, where $j\in\{1\leq l\leq N\colon U_l\cap\mathcal{N}_i^m\neq\varnothing\}$. By this construction, we have $\chi_i^{M,m}(x)=\sum_j\eta_j(x)=1$ for all $x\in\mathcal{N}_i^m$. 
By (1) and (3) in Assumption \ref{partition of unity}, it is obvious that $0\leq \chi_i^{m+2,m}(x)\leq1$ for any node $x\in\mathcal{N}_i^{m+2}\backslash\mathcal{N}_i^m$. Then we have $0\leq \chi_i^{M,m}(x)\leq1$ for any node $x\in\mathcal{N}_i^M\backslash\mathcal{N}_i^m$. Besides, $\chi_i^{M,m}=0\quad {\rm in}\hspace{0.5em}\mathcal{N}\backslash\mathcal{N}_i^M$ is also clear because of (3) in Assumption \ref{partition of unity}. 

The following lemma shows that our multiscale basis functions have a decay property.

\begin{lemma}
\label{basis decay property}
Consider the oversampling subgraphs $\mathcal{N}_i^l$ with $l\geq2$. Let $\phi_j^{i}\in V_\mathup{aux}$ be a given auxiliary multiscale basis function. Let $\psi_{j,l}^i$ be the multiscal basis functions obtained in (\ref{relaxed cem}) and let $\psi_j^i$ be the global multiscale basis functions obtained in (\ref{global relaxed cem}). Then it holds that 
$$
\norm{\psi_j^i-\psi_{j,l}^i}_a^2+\norm{\pi(\psi_j^i-\psi_{j,l}^i)}_s^2\leq c_{\star}\theta^{l'}\RoundBrackets*{\norm{\psi_j^i}_a^2+\norm{\pi \psi_j^i}_s^2},
$$
where $c_{\star}=5\big(F(1)^2C_{\eta}^2+1\big)\big(1+\frac{1}{\Lambda}\big)$, $\theta=\tilde{c}/(\tilde{c}+1)$ and $$\tilde{c}=18(F(1)+1)(C_\eta+1)\big(\frac{1}{\sqrt{\Lambda}}+1\big).$$ 
Besides, $l'=l/2$ if $l\geq2$ is an even number or  $l'=(l-1)/2$ if $l>2$ is an odd number.
\end{lemma} 
\begin{proof}
In terms of (\ref{variational_relaxed_cem}) and (\ref{global_variational_relaxed_cem}), we have (Let the oversampling subgraph considered in the variational problem be $\mathcal{N}_i^{l+2}$)
$$a(\psi_j^i-\psi_{j,l}^i,v)+s(\pi(\psi_j^i-\psi_{j,l}^i),\pi(v))=0$$
for all $v\in V(\mathcal{N}_i^{l+2})$. Taking $v=w-\psi_{j,l}^i$  with $w\in V(\mathcal{N}_i^{l+2})$ in above relation, we have
\begin{equation*}
	\begin{aligned}
		&\quad\norm{\psi_j^i-\psi_{j,l}^i}_a^2+\norm{\pi(\psi_j^i-\psi_{j,l}^i)}_s^2\\
  &=a(\psi_j^i-\psi_{j,l}^i,\psi_j^i-w+w-\psi_{j,l}^i)+s(\pi(\psi_j^i-\psi_{j,l}^i),\pi(\psi_j^i-w+w-\psi_{j,l}^i))\\
  &=a(\psi_j^i-\psi_{j,l}^i,\psi_j^i-w)+s(\pi(\psi_j^i-\psi_{j,l}^i),\pi(\psi_j^i-w))\\
  &\leq (\norm{\psi_j^i-\psi_{j,l}^i}_a^2+\norm{\psi_j^i-w}_a^2)/2+(\norm{\pi(\psi_j^i-\psi_{j,l}^i)}_s^2+\norm{\pi(\psi_j^i-w)}_s^2)/2
	\end{aligned}
\end{equation*}
Thus, we have
$$\norm{\psi_j^i-\psi_{j,l}^i}_a^2+\norm{\pi(\psi_j^i-\psi_{j,l}^i)}_s^2\leq \norm{\psi_j^i-w}_a^2+\norm{\pi(\psi_j^i-w)}_s^2$$
for all $w\in V(\mathcal{N}_i^{l+2})$. Letting $w=\chi_i^{l+2,l}\psi_j^i$ in above relation, we have
\begin{equation}
\label{local estimate}
\norm{\psi_j^i-\psi_{j,l}^i}_a^2+\norm{\pi(\psi_j^i-\psi_{j,l}^i)}_s^2\leq \norm{\psi_j^i-\chi_i^{l+2,l}\psi_j^i}_a^2+\norm{\pi(\psi_j^i-\chi_i^{l+2,l}\psi_j^i)}_s^2
\end{equation}
Next we will estimate these two terms on the right hand side of (\ref{local estimate}). We divide the proof into four steps.

\paragraph{\textbf{Step 1:}} We will estimate $\norm{(1-\chi_i^{l+2,l})\psi_j^i}_a^2$ in (\ref{local estimate}). By the definition of the norm $\norm{\cdot}_a$ and the fact that ${\rm supp}(1-\chi_i^{l+2,l})\subset(\mathcal{N}\backslash\mathcal{N}_i^l)$, we have
\begin{equation*}
	\begin{aligned}
		\norm{(1-\chi_i^{l+2,l})\psi_j^i}_a^2&\leq2\Big(\sum_{x\in\mathcal{N}\backslash\mathcal{N}_i^{l}}\frac{1}{2}\sum_{y\sim x}L_{xy}\big((1-\chi_i^{l+2,l}(x))\cdot\psi_j^i(x)-(1-\chi_i^{l+2,l}(y))\cdot\psi_j^i(y)\big)^2\Big)\\
  &\quad+\sum_{x\in\mathcal{N}\backslash\mathcal{N}_i^{l}}M_x\big((1-\chi_i^{l+2,l}(x))\cdot\psi_j^i(x)\big)^2\\
  &=2\Big(\sum_{x\in\mathcal{N}\backslash\mathcal{N}_i^{l}}\frac{1}{2}\sum_{y\sim x}L_{xy}\big((1-\chi_i^{l+2,l}(x))\cdot\psi_j^i(x)-(1-\chi_i^{l+2,l}(y))\cdot\psi_j^i(x)\\
  &\qquad +(1-\chi_i^{l+2,l}(y))\cdot\psi_j^i(x)-(1-\chi_i^{l+2,l}(y))\cdot\psi_j^i(y)\big)^2\Big)+\sum_{x\in\mathcal{N}\backslash\mathcal{N}_i^{l}}M_x\big((1-\chi_i^{l+2,l}(x))\cdot\psi_j^i(x)\big)^2\\
  &=2\Big(\sum_{x\in\mathcal{N}\backslash\mathcal{N}_i^{l}}\frac{1}{2}\sum_{y\sim x}L_{xy}\Big(\big((1-\chi_i^{l+2,l}(x))-(1-\chi_i^{l+2,l}(y))\big)\cdot\psi_j^i(x)\\
  &\qquad +(1-\chi_i^{l+2,l}(y))\cdot\big(\psi_j^i(x)-\psi_j^i(y)\big)\Big)^2\Big) +\sum_{x\in\mathcal{N}\backslash\mathcal{N}_i^{l}}M_x\big((1-\chi_i^{l+2,l}(x))\cdot\psi_j^i(x)\big)^2\\
  &\leq4\bigg(\sum_{x\in\mathcal{N}\backslash\mathcal{N}_i^{l}}\frac{1}{2}\sum_{y\sim x}L_{xy}\Big((\chi_i^{l+2,l}(x)-\chi_i^{l+2,l}(y))^2\cdot\big(\psi_j^i(x)\big)^2\\
  &\qquad +\big(1-\chi_i^{l+2,l}(y)\big)^2\cdot\big(\psi_j^i(x)-\psi_j^i(y)\big)^2\Big)\bigg) +\sum_{x\in\mathcal{N}\backslash\mathcal{N}_i^{l}}M_x\big((1-\chi_i^{l+2,l}(x))\cdot\psi_j^i(x)\big)^2
	\end{aligned}
\end{equation*}
By (\ref{cutoff function_c}), we know
\begin{equation}
\label{ka_1}
\abs{\chi_i^{l+2,l}}\leq1,\quad\abs{1-\chi_i^{l+2,l}}\leq1.
\end{equation}
By (1) of Assumption \ref{assumption on partitions} and  (2), (3) of Assumption \ref{partition of unity}, we easily obtain
\begin{equation}
\label{ka_3}
\abs{\chi_i^{l+2,l}(x)-\chi_i^{l+2,l}(y)}\leq F(1)C_{\eta}(C_{\mathrm{po}})^{-1}
\end{equation}
\begin{equation}
\label{ka_2}
\abs{(1-\chi_i^{l+2,l}(x))-(1-\chi_i^{l+2,l}(y))}\leq F(1)C_{\eta}(C_{\mathrm{po}})^{-1}
\end{equation}
for all $x$ belonging to some $\mathcal{N}_i$ and $y\sim x$.
Then by (1), (2) of Assumption \ref{partition of unity}, we have
\begin{equation}
\label{ka_estimate}
	\begin{aligned}
\norm{(1-\chi_i^{l+2,l})\psi_j^i}_a^2&\leq4\Big(\sum_{x\in\mathcal{N}\backslash\mathcal{N}_i^{l}}\frac{1}{2}\sum_{y\sim x}L_{xy}F(1)^2C_{\eta}^2(C_{\mathrm{po}})^{-2}\big(\psi_j^i(x)\big)^2\\
  &\qquad +L_{xy}\big(\psi_j^i(x)-\psi_j^i(y)\big)^2\Big)+\sum_{x\in\mathcal{N}\backslash\mathcal{N}_i^{l}}M_x\big(\psi_j^i(x)\big)^2\\
  &\leq 4\big(F(1)^2C_{\eta}^2+1\big)(\norm{\psi_j^i}_{s(\mathcal{N}\backslash\mathcal{N}_i^{l})}^2+\norm{\psi_j^i}_{a(\mathcal{N}\backslash\mathcal{N}_i^{l})}^2).
 \end{aligned}
\end{equation}
We note that for each $\mathcal{N}_k(1\leq k\leq N)$, combining (\ref{basic estimates for pi}),  we have 
\begin{equation}
\begin{split}
\label{psi_local}
\norm{\psi_j^i}_{s(\mathcal{N}_k)}^2&=\norm{(I-\pi)(\psi_j^i)+\pi(\psi_j^i)}_{s(\mathcal{N}_k)}^2=\norm{(I-\pi)(\psi_j^i)}_{s(\mathcal{N}_k)}^2+\norm{\pi(\psi_j^i)}_{s(\mathcal{N}_k)}^2\\
&\leq \frac{1}{\Lambda}\norm{\psi_j^i}_{a_k}^2+\norm{\pi(\psi_j^i)}_{s(\mathcal{N}_k)}^2\leq \frac{1}{\Lambda}\norm{\psi_j^i}_{a(\mathcal{N}_k)}^2+\norm{\pi(\psi_j^i)}_{s(\mathcal{N}_k)}^2.
\end{split}
\end{equation}
In terms of (\ref{comparison of a}), we have
\begin{equation*}
\label{compatison sum_a}
\sum_{\mathcal{N}_k\subset\mathcal{N}\backslash\mathcal{N}_i^{l}}\norm{\psi_j^i}_{a_k}^2=\sum_{\mathcal{N}_k\subset\mathcal{N}\backslash\mathcal{N}_i^{l}}a_k(\psi_j^i,\psi_j^i)\leq \norm{\psi_j^i}_{a(\mathcal{N}\backslash\mathcal{N}_i^{l})}^2.
\end{equation*}
Using this inequality into (\ref{psi_local}), we obtain
\begin{equation}
\label{psi_set}
\norm{\psi_j^i}_{s(\mathcal{N}\backslash\mathcal{N}_i^{l})}^2=\sum_{\mathcal{N}_k\subset(\mathcal{N}\backslash\mathcal{N}_i^{(l-2)})}\norm{\psi_j^i}_{s(\mathcal{N}_k)}^2\leq\frac{1}{\Lambda}\norm{\psi_j^i}_{a(\mathcal{N}\backslash\mathcal{N}_i^{l})}^2+\norm{\pi(\psi_j^i)}_{s(\mathcal{N}\backslash\mathcal{N}_i^{l})}^2.
\end{equation}
Substituting (\ref{psi_set}) into (\ref{ka_estimate}), we have
\begin{equation}
\label{Step 1 estimate}
\norm{(1-\chi_i^{l+2,l})\psi_j^i}_a^2\leq 4\big(F(1)^2C_{\eta}^2+1\big)\big(\big(1+\frac{1}{\Lambda}\big)\norm{\psi_j^i}_{a(\mathcal{N}\backslash\mathcal{N}_i^{l})}^2+\norm{\pi(\psi_j^i)}_{s(\mathcal{N}\backslash\mathcal{N}_i^{l})}^2\big).
\end{equation}

\paragraph{\textbf{Step 2:}} We will estimate the term $\norm{\pi(\psi_j^i-\chi_i^{l+2,l}\psi_j^i)}_s^2$ in (\ref{local estimate}). By the estimate (\ref{basic estimates for pi_2}) and the fact that $\abs{1-\chi_i^{l+2,l}}\leq1$, $1-\chi_i^{l+2,l}\equiv0$ in $\mathcal{N}_i^{l}$, we have
$$\norm{\pi(\psi_j^i-\chi_i^{l+2,l}\psi_j^i)}_s^2\leq\norm{\psi_j^i-\chi_i^{l+2,l}\psi_j^i}_s^2\leq\norm{\psi_j^i}_{s(\mathcal{N}\backslash\mathcal{N}_i^{l})}^2.$$
By utilizing (\ref{psi_set}), we have
$$\norm{\pi(\psi_j^i-\chi_i^{l+2,l}\psi_j^i)}_s^2\leq\frac{1}{\Lambda}\norm{\psi_j^i}_{a(\mathcal{N}\backslash\mathcal{N}_i^{l})}^2+\norm{\pi(\psi_j^i)}_{s(\mathcal{N}\backslash\mathcal{N}_i^{l})}^2.$$
In term of Steps 1 and 2, (\ref{local estimate}) can be estimated as
\begin{equation}
\label{local estimate2}
\norm{\psi_j^i-\psi_{j,l}^i}_a^2+\norm{\pi(\psi_j^i-\psi_{j,l}^i)}_s^2\leq5\big(F(1)^2C_{\eta}^2+1\big)\big(1+\frac{1}{\Lambda}\big)\big(\norm{\psi_j^i}_{a(\mathcal{N}\backslash\mathcal{N}_i^{l})}^2+\norm{\pi(\psi_j^i)}_{s(\mathcal{N}\backslash\mathcal{N}_i^{l})}^2\big).
\end{equation}
Next we estimate the right hand side of (\ref{local estimate2}).

\paragraph{\textbf{Step 3:}} We will show that $\norm{\psi_j^i}_{a(\mathcal{N}\backslash\mathcal{N}_i^{l})}^2+\norm{\pi(\psi_j^i)}_{s(\mathcal{N}\backslash\mathcal{N}_i^{l})}^2$ can be estimated by $\norm{\psi_j^i}_{a(\mathcal{N}_i^{l}\backslash\mathcal{N}_i^{l-2})}^2+\norm{\pi(\psi_j^i)}_{s(\mathcal{N}_i^{l}\backslash\mathcal{N}_i^{l-2})}^2$. By utilizing (\ref{global_variational_relaxed_cem}) and the test function $2\xi^2\psi_j^i$,  where $\xi_x=1-\chi_i^{l,l-2}(x)$ for $x\in\mathcal{N}$, we have
\begin{equation}
\label{rex_variational_orthogonal}
a(\psi_j^i,2\xi^2\psi_j^i)+s\big(\pi(\psi_j^i),\pi(2\xi^2\psi_j^i)\big)=s\big(\phi^i_j,\pi(2\xi^2\psi_j^i)\big)=0
\end{equation}
where the last equality follows from the facts that supp$(1-\chi_i^{l,l-2})\subset(\mathcal{N}\backslash\mathcal{N}_i^{l-2})$ and supp$(\phi^i_j)\subset\mathcal{N}_i$.

Note that
\begin{equation}
	\label{set less}
	\begin{split}
		\norm{\psi_j^i}_{a(\mathcal{N}\backslash\mathcal{N}_i^{l})}^2&\leq \sum_{\mathcal{N},\sim}L_{xy}\big(\psi_j^i(x)-\psi_j^i(y)\big)^2\cdot2\big(\xi_x^2+\xi_y^2-\xi_x\cdot\xi_y\big)+\sum_{x\in\mathcal{N}\backslash\mathcal{N}_i^{l-2}}M_x\cdot\psi_j^i(x)^2\cdot2\xi_x^2\\
  &=T_1+T_2,
	\end{split}
\end{equation}
where $T_1=\sum_{\mathcal{N},\sim}L_{xy}\big(\psi_j^i(x)-\psi_j^i(y)\big)^2\cdot2\big(\xi_x^2+\xi_y^2-\xi_x\cdot\xi_y\big)$ and $T_2=\sum_{x\in\mathcal{N}\backslash\mathcal{N}_i^{l-2}}M_x\cdot\psi_j^i(x)^2\cdot2\xi_x^2$.

For $T_1$, by combining the estimates (\ref{ka_1}), (2) of Assumption \ref{partition of unity} and (\ref{ka_2}), we have 
	\begin{align*}
		T_1&=\sum_{\mathcal{N},\sim}L_{xy}\big(\psi_j^i(x)-\psi_j^i(y)\big)\cdot2\big(\xi_x^2\psi_j^i(x)-\xi_y^2\psi_j^i(y)\big)\\
  &\quad+\sum_{\mathcal{N},\sim}L_{xy}\big(\psi_j^i(x)-\psi_j^i(y)\big)\cdot2\bigg(-\xi_x\psi_j^i(y)\cdot\big(\xi_x-\xi_y\big)-\xi_y\psi_j^i(x)\cdot\big(\xi_x-\xi_y\big)\bigg)\\
  &\leq \sum_{\mathcal{N},\sim}L_{xy}\big(\psi_j^i(x)-\psi_j^i(y)\big)\cdot2\big(\xi_x^2\psi_j^i(x)-\xi_y^2\psi_j^i(y)\big)\\
  &\quad+2\sum_{x\in\mathcal{N}_i^{l}\backslash\mathcal{N}_i^{l-2}}\frac{1}{2}\sum_{y\sim x}L_{xy}\abs{\psi_j^i(x)-\psi_j^i(y)}\cdot2\cdot\abs{-\xi_x\psi_j^i(y)\cdot(\xi_x-\xi_y)-\xi_y\psi_j^i(x)\cdot(\xi_x-\xi_y)}\\
  &=\sum_{\mathcal{N},\sim}L_{xy}\big(\psi_j^i(x)-\psi_j^i(y)\big)\cdot2\big(\xi_x^2\psi_j^i(x)-\xi_y^2\psi_j^i(y)\big)\\
  &\quad+2\sum_{x\in\mathcal{N}_i^{l}\backslash\mathcal{N}_i^{l-2}}\frac{1}{2}\sum_{y\sim x}L_{xy}\abs{\psi_j^i(x)-\psi_j^i(y)}\cdot2\cdot\abs{\xi_y-\xi_x}\cdot\abs{\xi_x\cdot(\psi_j^i(y)-\psi_j^i(x))+\psi_j^i(x)\cdot(\xi_x+\xi_y)}\\
  &\leq \sum_{\mathcal{N},\sim}L_{xy}\big(\psi_j^i(x)-\psi_j^i(y)\big)\cdot2\big(\xi_x^2\psi_j^i(x)-\xi_y^2\psi_j^i(y)\big)+2\sum_{x\in\mathcal{N}_i^{l}\backslash\mathcal{N}_i^{l-2}}\frac{1}{2}\sum_{y\sim x}4L_{xy}\big(\psi_j^i(x)-\psi_j^i(y)\big)^2\\
  &\quad+2\sum_{x\in\mathcal{N}_i^{l}\backslash\mathcal{N}_i^{l-2}}\frac{1}{2}\sum_{y\sim x}4L_{xy}F(1)C_{\eta}(C_{\mathrm{po}})^{-1}\cdot\abs{\psi_j^i(x)-\psi_j^i(y)}\cdot\abs{\psi_j^i(x)}.
	\end{align*}
By using Holder's inequality for the last term of the right hand side of above inequality, we have
	\begin{align*}
		T_1&\leq \sum_{\mathcal{N},\sim}L_{xy}\big(\psi_j^i(x)-\psi_j^i(y)\big)\cdot2\big(\xi_x^2\psi_j^i(x)-\xi_y^2\psi_j^i(y)\big)+8\norm{\psi_j^i}_{a(\mathcal{N}_i^{l}\backslash\mathcal{N}_i^{l-2})}^2\\
&\quad+8F(1)C_{\eta}\Big(\sum_{x\in\mathcal{N}_i^{l}\backslash\mathcal{N}_i^{l-2}}\frac{1}{2}\sum_{y\sim x}L_{xy}\big(\psi_j^i(x)-\psi_j^i(y)\big)^2\Big)^{\frac{1}{2}}\cdot\Big(\sum_{x\in\mathcal{N}_i^{l}\backslash\mathcal{N}_i^{l-2}}\frac{1}{2}\sum_{y\sim x}L_{xy}(C_{\mathrm{po}})^{-2}\cdot\big(\psi_j^i(x)\big)^2\Big)^{\frac{1}{2}}\\
  &\leq \sum_{\mathcal{N},\sim}L_{xy}\big(\psi_j^i(x)-\psi_j^i(y)\big)\cdot2\big(\xi_x^2\psi_j^i(x)-\xi_y^2\psi_j^i(y)\big)+8\norm{\psi_j^i}_{a(\mathcal{N}_i^{l}\backslash\mathcal{N}_i^{l-2})}^2\\
&\quad+8F(1)C_{\eta}\norm{\psi_j^i}_{a(\mathcal{N}_i^{l}\backslash\mathcal{N}_i^{l-2})}\norm{\psi_j^i}_{s(\mathcal{N}_i^{l}\backslash\mathcal{N}_i^{l-2})}.
	\end{align*}
In term of (\ref{psi_set}), we clearly obtain
\begin{equation*}
\norm{\psi_j^i}_{s(\mathcal{N}_i^{l}\backslash\mathcal{N}_i^{l-2})}^2=\sum_{\mathcal{N}_k\subset(\mathcal{N}_i^{l}\backslash\mathcal{N}_i^{l-2})}\norm{\psi_j^i}_{s(\mathcal{N}_k)}^2\leq\frac{1}{\Lambda}\norm{\psi_j^i}_{a(\mathcal{N}_i^{l}\backslash\mathcal{N}_i^{l-2})}^2+\norm{\pi(\psi_j^i)}_{s(\mathcal{N}_i^{l}\backslash\mathcal{N}_i^{l-2})}^2.
\end{equation*}
Then we obtain
\begin{equation}
\label{T1_estimate}
	\begin{split}
		T_1&\leq \sum_{\mathcal{N},\sim}L_{xy}\big(\psi_j^i(x)-\psi_j^i(y)\big)\cdot2\big(\xi_x^2\psi_j^i(x)-\xi_y^2\psi_j^i(y)\big)+8\norm{\psi_j^i}_{a(\mathcal{N}_i^{l}\backslash\mathcal{N}_i^{l-2})}^2\\
    &\quad+8F(1)C_{\eta}\norm{\psi_j^i}_{a(\mathcal{N}_i^{l}\backslash\mathcal{N}_i^{l-2})}\big(\frac{1}{\sqrt{\Lambda}}\norm{\psi_j^i}_{a(\mathcal{N}_i^{l}\backslash\mathcal{N}_i^{l-2})}+\norm{\pi(\psi_j^i)}_{s(\mathcal{N}_i^{l}\backslash\mathcal{N}_i^{l-2})}\big)\\
  &\leq \sum_{\mathcal{N},\sim}L_{xy}\big(\psi_j^i(x)-\psi_j^i(y)\big)\cdot2\big(\xi_x^2\psi_j^i(x)-\xi_y^2\psi_j^i(y)\big)+8\norm{\psi_j^i}_{a(\mathcal{N}_i^{l}\backslash\mathcal{N}_i^{l-2})}^2\\
  &\quad+8F(1)C_{\eta}\big(\frac{1}{\sqrt{\Lambda}}+\frac{1}{2}\big)\big(\norm{\psi_j^i}_{a(\mathcal{N}_i^{l}\backslash\mathcal{N}_i^{l-2})}^2+\norm{\pi(\psi_j^i)}_{s(\mathcal{N}_i^{l}\backslash\mathcal{N}_i^{l-2})}^2\big)\\
  &\leq \sum_{\mathcal{N},\sim}L_{xy}\big(\psi_j^i(x)-\psi_j^i(y)\big)\cdot2\big(\xi_x^2\psi_j^i(x)-\xi_y^2\psi_j^i(y)\big)\\
  &\quad+16(F(1)+1)(C_{\eta}+1)\big(\frac{1}{\sqrt{\Lambda}}+1\big)\big(\norm{\psi_j^i}_{a(\mathcal{N}_i^{l}\backslash\mathcal{N}_i^{l-2})}^2+\norm{\pi(\psi_j^i)}_{s(\mathcal{N}_i^{l}\backslash\mathcal{N}_i^{l-2})}^2\big).
	\end{split}
\end{equation}
Adding $T_2$ to both sides of (\ref{T1_estimate}) and using the fact that $\xi_x=0$ in $\mathcal{N}_i^{l-2}$, we get
\begin{equation}
\label{estimate_psi_a_norm}
	\begin{split}
		\norm{\psi_j^i}_{a(\mathcal{N}\backslash\mathcal{N}_i^{l})}^2&\leq \sum_{\mathcal{N},\sim}L_{xy}\big(\psi_j^i(x)-\psi_j^i(y)\big)\cdot2\big(\xi_x^2\psi_j^i(x)-\xi_y^2\psi_j^i(y)\big)+\sum_{x\in\mathcal{N}\backslash\mathcal{N}_i^{l-2}}M_x\cdot\psi_j^i(x)^2\cdot2\xi_x^2\\
  &\quad+16(F(1)+1)(C_{\eta}+1)\big(\frac{1}{\sqrt{\Lambda}}+1\big)\big(\norm{\psi_j^i}_{a(\mathcal{N}_i^{l}\backslash\mathcal{N}_i^{l-2})}^2+\norm{\pi(\psi_j^i)}_{s(\mathcal{N}_i^{l}\backslash\mathcal{N}_i^{l-2})}^2\big)\\
  &=a(\psi_{j}^{i},2\xi^2\psi_{j}^{i})+16(F(1)+1)(C_{\eta}+1)\big(\frac{1}{\sqrt{\Lambda}}+1\big)\big(\norm{\psi_j^i}_{a(\mathcal{N}_i^{l}\backslash\mathcal{N}_i^{l-2})}^2+\norm{\pi(\psi_j^i)}_{s(\mathcal{N}_i^{l}\backslash\mathcal{N}_i^{l-2})}^2\big).
	\end{split}
\end{equation}
Since $\xi_x=1$ for any $x\in\mathcal{N}\backslash\mathcal{N}_i^{l}$ and $\xi_x=0$ for any $x\in\mathcal{N}_i^{l-2}$, we have
\begin{align*}
s(\pi(\psi_j^i),2\xi^2\psi_j^i)&=s\big(\pi(\psi_j^i),\pi(2\xi^2\psi_j^i)\big)=2\norm{\pi(\psi_j^i)}_{s(\mathcal{N}\backslash\mathcal{N}_i^{l})}^2\\
&\quad+\sum_{x\in\mathcal{N}_i^{l}\backslash\mathcal{N}_i^{l-2}}\big(\frac{1}{2}\sum_{y\sim x}L_{xy}+M_x\big)(C_{\mathrm{po}})^{-2}2\xi_x^2\cdot\pi(\psi_j^i)(x)\cdot\psi_j^i(x).
\end{align*}
By above equality, the fact that $\abs{\xi}\leq1$ and holder's inequality, we obtain 
\begin{align*}
\norm{\pi(\psi_j^i)}_{s(\mathcal{N}\backslash\mathcal{N}_i^{l})}^2&<2\norm{\pi(\psi_j^i)}_{s(\mathcal{N}\backslash\mathcal{N}_i^{l})}^2=s\big(\pi(\psi_j^i),\pi(2\xi^2\psi_j^i)\big)\\
&\qquad-\sum_{x\in\mathcal{N}_i^{l}\backslash\mathcal{N}_i^{l-2}}\big(\frac{1}{2}\sum_{y\sim x}L_{xy}+M_x\big)(C_{\mathrm{po}})^{-2}2\xi_x^2\cdot\pi(\psi_j^i)(x)\cdot\psi_j^i(x)\\
&\leq s\big(\pi(\psi_j^i),\pi(2\xi^2\psi_j^i)\big)+2\norm{\pi(\psi_j^i)}_{s(\mathcal{N}_i^{l}\backslash\mathcal{N}_i^{l-2})}\cdot\norm{\psi_j^i}_{s(\mathcal{N}_i^{l}\backslash\mathcal{N}_i^{l-2})}.
\end{align*}
In terms of (\ref{psi_set}), we have 
$$\norm{\psi_j^i}_{s(\mathcal{N}_i^{l}\backslash\mathcal{N}_i^{l-2})}\leq\frac{1}{\sqrt{\Lambda}}\norm{\psi_j^i}_{a(\mathcal{N}_i^{l}\backslash\mathcal{N}_i^{l-2})}+\norm{\pi(\psi_j^i)}_{s(\mathcal{N}_i^{l}\backslash\mathcal{N}_i^{l-2})}.$$
Combining above two inequalities, we get
\begin{equation}
\label{estimate_pi_psi_s}
	\begin{aligned}
\norm{\pi(\psi_j^i)}_{s(\mathcal{N}\backslash\mathcal{N}_i^{l})}^2&\leq s\big(\pi(\psi_j^i),\pi(2\xi^2\psi_j^i)\big)+2\norm{\pi(\psi_j^i)}_{s(\mathcal{N}_i^{l}\backslash\mathcal{N}_i^{l-2})}\cdot\big(\frac{1}{\sqrt{\Lambda}}\norm{\psi_j^i}_{a(\mathcal{N}_i^{l}\backslash\mathcal{N}_i^{l-2})}+\norm{\pi(\psi_j^i)}_{s(\mathcal{N}_i^{l}\backslash\mathcal{N}_i^{l-2})}\big)\\
&\leq s\big(\pi(\psi_j^i),\pi(2\xi^2\psi_j^i)\big)+2\big(\frac{1}{\sqrt{\Lambda}}+1\big)\big(\norm{\psi_j^i}_{a(\mathcal{N}_i^{l}\backslash\mathcal{N}_i^{l-2})}^2+\norm{\pi(\psi_j^i)}_{s(\mathcal{N}_i^{l}\backslash\mathcal{N}_i^{l-2})}^2\big).
\end{aligned}
\end{equation}
By (\ref{rex_variational_orthogonal}), (\ref{estimate_psi_a_norm}) and (\ref{estimate_pi_psi_s}), we obtain that
\begin{equation}
\label{estimate_in_step_3}
\begin{aligned}
\norm{\psi_j^i}_{a(\mathcal{N}\backslash\mathcal{N}_i^{l})}^2+\norm{\pi(\psi_j^i)}_{s(\mathcal{N}\backslash\mathcal{N}_i^{l})}^2&\leq18(F(1)+1)(C_{\eta}+1)\big(\frac{1}{\sqrt{\Lambda}}+1\big)\big(\norm{\psi_j^i}_{a(\mathcal{N}_i^{l}\backslash\mathcal{N}_i^{l-2})}^2+\norm{\pi(\psi_j^i)}_{s(\mathcal{N}_i^{l}\backslash\mathcal{N}_i^{l-2})}^2\big).
\end{aligned}
\end{equation}

\paragraph{\textbf{Step 4:}} We will show that $\norm{\psi_j^i}_{a(\mathcal{N}\backslash\mathcal{N}_i^{l})}^2+\norm{\pi(\psi_j^i)}_{s(\mathcal{N}\backslash\mathcal{N}_i^{l})}^2$ can be estimated by $\norm{\psi_j^i}_{a(\mathcal{N}\backslash\mathcal{N}_i^{l-2})}^2+\norm{\pi(\psi_j^i)}_{s(\mathcal{N}\backslash\mathcal{N}_i^{l-2})}^2$. 
\begin{align*}
&\quad\norm{\psi_j^i}_{a(\mathcal{N}\backslash\mathcal{N}_i^{l-2})}^2+\norm{\pi(\psi_j^i)}_{s(\mathcal{N}\backslash\mathcal{N}_i^{l-2})}^2\\
&=\norm{\psi_j^i}_{a(\mathcal{N}\backslash\mathcal{N}_i^{l})}^2+\norm{\pi(\psi_j^i)}_{s(\mathcal{N}\backslash\mathcal{N}_i^{l})}^2+\norm{\psi_j^i}_{a(\mathcal{N}_i^{l}\backslash\mathcal{N}_i^{l-2})}^2+\norm{\pi(\psi_j^i)}_{s(\mathcal{N}_i^{l}\backslash\mathcal{N}_i^{l-2})}^2\\
&\geq\Big(1+\big(18(F(1)+1)(C_{\eta}+1)\big(\frac{1}{\sqrt{\Lambda}}+1\big)\big)^{-1}\Big)\big(\norm{\psi_j^i}_{a(\mathcal{N}\backslash\mathcal{N}_i^{l})}^2+\norm{\pi(\psi_j^i)}_{s(\mathcal{N}\backslash\mathcal{N}_i^{l})}^2\big),
\end{align*}
where (\ref{estimate_in_step_3}) is used in the last inequality. Utilizing the above inequality recursively and if $l\geq2$ is an even number, we have
$$\norm{\psi_j^i}_{a(\mathcal{N}\backslash\mathcal{N}_i^{l})}^2+\norm{\pi(\psi_j^i)}_{s(\mathcal{N}\backslash\mathcal{N}_i^{l})}^2\leq\Big(1+\big(18(F(1)+1)(C_{\eta}+1)\big(\frac{1}{\sqrt{\Lambda}}+1\big)\big)^{-1}\Big)^{-\frac{l}{2}}(\norm{\psi_j^i}_a^2+\norm{\pi(\psi_j^i)}_s^2).$$
If $l>2$ is an odd number, we have
$$\norm{\psi_j^i}_{a(\mathcal{N}\backslash\mathcal{N}_i^{l})}^2+\norm{\pi(\psi_j^i)}_{s(\mathcal{N}\backslash\mathcal{N}_i^{l})}^2\leq\Big(1+\big(18(F(1)+1)(C_{\eta}+1)\big(\frac{1}{\sqrt{\Lambda}}+1\big)\big)^{-1}\Big)^{\frac{1-l}{2}}(\norm{\psi_j^i}_a^2+\norm{\pi(\psi_j^i)}_s^2).$$
The proof is completed.
\end{proof}

\begin{lemma}
\label{global and multiscale}
With the same notations in Lemma \ref{basis decay property}, we can get
$$\norm{\sum_{i=1}^N(\psi_j^i-\psi_{j,l}^i)}_a^2\leq C_\mathup{ol}F(l+2)\sum_{i=1}^N\norm{\psi_j^i-\psi_{j,l}^i}_a^2,$$
where $C_\mathup{ol}>0$ does not depend on $C_{\mathrm{po}},C_{\eta},l,\Lambda$. 
\end{lemma}
\begin{proof}
The proof is similar to \cite[Lemma 4]{Eric2018} so we omit the details here.
\end{proof}

Next, we state and prove the convergence of the multiscale solution $u_\mathup{ms}$ to $u$.

\begin{theorem}
\label{convergence analysis}
Let $u$ be the solution of (\ref{model problem}) and $u_\mathup{ms}$ be the multiscale solution of (\ref{multiscale problem}) . Then we have
$$\norm{u-u_\mathup{ms}}_a\leq\frac{1}{\sqrt{\Lambda}}\norm{f}_{s^*}+C_\mathup{ol}^{\frac{1}{2}}F(l+2)^{\frac{1}{2}}(c_{\star}\theta^{l'})^{\frac{1}{2}}(\norm{u_\mathup{glo}}_a+\norm{\pi(u_\mathup{glo})}_s).$$
where $u_\mathup{glo}\in V^\mathup{glo}_\mathup{ms}$ is the multiscale solution using global basis and the constant $C_\mathup{ol}>0$ does not depend on $C_{\mathrm{po}},C_{\eta},l,\Lambda$. Besides, $c_{\star},\theta,l'$ are defined the same as in Lemma \ref{basis decay property}.

\end{theorem}
\begin{proof}
By the variational form (\ref{global_variational_relaxed_cem}), we have 
\begin{equation*}
a(\psi_{j}^{i},\psi_{j}^{i})+s\big(\pi(\psi_{j}^{i}),\pi(\psi_{j}^{i})\big)=s\big(\phi^i_j,\pi(\psi_{j}^{i})\big)
\end{equation*}
By the above equality, on one hand, we have
$\norm{\pi(\psi_{j}^i)}_s^2\leq\norm{\phi^i_j}_s\cdot\norm{\pi(\psi_{j}^i)}_s$. That is,
\begin{equation}
\label{con_11}
\norm{\pi(\psi_{j}^i)}_s^2\leq\norm{\phi^i_j}_s^2.
\end{equation}
On the other hand, 
\begin{align*}
s\big(\phi^i_j,\pi(\psi_{j}^{i})\big)\leq\frac{1}{\sqrt{2}}\norm{\phi^i_j}_s\cdot\sqrt{2}\norm{\pi(\psi_{j}^i)}_s\leq\big(\frac{1}{2}\norm{\phi^i_j}_s^2+2\norm{\pi(\psi_{j}^i)}_s^2\big)/2=\frac{1}{4}\norm{\phi^i_j}_s^2+s\big(\pi(\psi_{j}^{i}),\pi(\psi_{j}^{i})\big).
\end{align*}
Then, we have
\begin{equation}
\label{con_22}
\norm{\psi_{j}^i}_a^2\leq\frac{1}{4}\norm{\phi^i_j}_s^2.
\end{equation}
By combining (\ref{con_11}), (\ref{con_22}) and Lemma \ref{basis decay property}, we obtain
\begin{equation}
\label{decay_2}
\norm{\psi_j^i-\psi_{j,l}^i}_a^2+\norm{\pi(\psi_j^i-\psi_{j,l}^i)}_s^2\leq \frac{5}{4}c_{\star}\theta^{l'}\norm{\phi_j^i}_s^2.
\end{equation}
We write $u_\mathup{glo}=\sum_{i=1}^N\sum_{j=1}^{l_i}c^i_j\psi_j^i$ and define $w=\sum_{i=1}^N\sum_{j=1}^{l_i}c^i_j\psi_{j,l}^i\in V_\mathup{ms}^m$. By the orthogonal property, we have (Note that $w-u_\mathup{ms}\in V_\mathup{ms}^m$)
\begin{align*}
\norm{u-u_\mathup{ms}}_a^2&=a(u-u_\mathup{ms},u-u_\mathup{ms})=a(u-u_\mathup{ms},u-w+w-u_\mathup{ms})\\
&=a(u-u_\mathup{ms},u-w)\leq \norm{u-u_\mathup{ms}}_a\cdot\norm{u-w}_a.
\end{align*}
Then we have
$$\norm{u-u_\mathup{ms}}_a\leq\norm{u-w}_a\leq\norm{u-u_\mathup{glo}}_a+\norm{\sum_{i=1}^N\sum_{j=1}^{l_i}c^i_j(\psi_j^i-\psi_{j,l}^i)}_a.$$
By utilizing Lemmas \ref{basis decay property},  \ref{global and multiscale} and (\ref{decay_2})  $\big($applying them to the function $\sum_{j=1}^{l_i}c^i_j(\psi_j^i-\psi_{j,l}^i)$$\big)$ and denoting $\phi=\sum_{i=1}^N\sum_{j=1}^{l_i}c^i_j\phi^i_j\in V_\mathup{aux}$, we obtain
\begin{align*}
\norm{w-u_\mathup{glo}}_a^2&\leq C_\mathup{ol}F(l+2)\sum_{i=1}^N\norm{\sum_{j=1}^{l_i}c^i_j(\psi_j^i-\psi_{j,l}^i)}_a^2\\
&\leq C_\mathup{ol}F(l+2)\cdot\frac{5}{4}c_{\star}\theta^{l'}\sum_{i=1}^N\sum_{j=1}^{l_i}\norm{c^i_j\phi_j^{i}}_s^2\\
&=C_\mathup{ol}F(l+2)\cdot\frac{5}{4}c_{\star}\theta^{l'}\sum_{i=1}^N\sum_{j=1}^{l_i}(c^i_j)^2=CF(l+2)\cdot\frac{5}{4}c_{\star}\theta^{l'}\cdot s(\phi,\phi),
\end{align*}
where the orthogonality of the eigenfunctions $\phi^i_j$ of (\ref{spectral problem}) is used in the last two lines. 
By the definitions of $u_\mathup{glo},\phi$ and the variational form (\ref{global_variational_relaxed_cem}), we know 
\begin{equation}
\label{con_33}
a(u_\mathup{glo},v)+s(\pi(u_\mathup{glo}),\pi(v))=s(\phi,\pi(v)),\quad \forall v\in V.
\end{equation}
For $\phi\in V_\mathup{aux}$, we clearly have
\begin{equation}
\label{a_equiv_s}
\begin{split}
\norm{\phi}_a^2&=\sum_{\mathcal{N},\sim}L_{xy}\big(\phi_x-\phi_y\big)^2+\sum_{x\in\mathcal{N}}M_x\leq2\sum_{\mathcal{N},\sim}L_{xy}\big(\phi_x^2+\phi_y^2\big)+\sum_{x\in\mathcal{N}}M_x\phi_x^2\\
&\leq4\Big(\sum_{\mathcal{N},\sim}L_{xy}\big(\phi_x\big)^2+\sum_{x\in\mathcal{N}}M_x\phi_x^2\Big)=4\norm{\phi}_s^2.
\end{split}
\end{equation}
Letting $v=\phi$ in (\ref{con_33}), we have $$a(u_\mathup{glo},\phi)+s(\pi(u_\mathup{glo}),\pi(\phi))=s(\phi,\pi(\phi))=s(\phi,\phi).$$
Then, combining (\ref{a_equiv_s}), we obtain
\begin{align*}
s(\phi,\phi)&=a(u_\mathup{glo},\phi)+s(\pi(u_\mathup{glo}),\phi)\leq\norm{u_\mathup{glo}}_a\norm{\phi}_a+\norm{\pi(u_\mathup{glo})}_s\norm{\phi}_s\\
&\leq2\norm{\phi}_s(\norm{u_\mathup{glo}}_a+\norm{\pi(u_\mathup{glo})}_s).
\end{align*}
Therefore, we have
$$\norm{w-u_\mathup{glo}}_a\leq C_\mathup{ol}^{\frac{1}{2}}F(l+2)^{1/2}(c_{\star}\theta^{l'})^{1/2}(\norm{u_\mathup{glo}}_a+\norm{\pi(u_\mathup{glo})}_s),$$
where $C_\mathup{ol}>0$ does not depend on $C_{\mathrm{po}},C_{\eta},l,\Lambda$. 
Combining Theorem \ref{global convergence}, we obtain the desired estimate.
\end{proof}
\begin{remark}
We can give more details from this theorem. It is easy to find that $\norm{f}_{s^*}=O(C_{\mathrm{po}})$. Recall (\ref{u_glo}), it follows that $\norm{u_\mathup{glo}}_a=O(1)$. We also have $\norm{\pi(u_\mathup{glo})}_s=\norm{\pi(u)}_s=O\big((C_{\mathrm{po}})^{-1}\big)$. Therefore, by choosing $l$ such that $(l+1)^{d/2}(c_{\star}\theta^{l'})^{1/2}=O\big((C_{\mathrm{po}})^2\big)$, we have
$$\norm{u-u_\mathup{ms}}_a=O(C_{\mathrm{po}}).$$
\end{remark}
\subsection{On the number of small contrast-dependent eigenvectors}\label{number of small eigenvalues}
In this section, we show that the number of small eigenvalues is determined by the number of disconnected high-contrast inclusions or channels, and that the smallest contrast-independent eigenvalue is of order one. To simplify the presentation, we temporarily omit the factor $(C_{\mathrm{po}})^{-2}$ on the right-hand side of the eigenvalue problem.

Denote $\mathcal{R}_1$ as the region not containing high-contrast structures. We set $L_{xy}=1$ for all edges with $x\in\mathcal{R}_1$ (note that $y$ may lie outside $\mathcal{R}_1$). The high-contrast structure is characterized by a parameter $\xi \gg1$. We decompose the node set as $\mathcal{N}=\mathcal{R}_1\cup\mathcal{R}^{\xi}$, where $\mathcal{R}^{\xi}=\cup_{m=1}^{N_\mathup{ch}}\mathcal{R}^{\xi}_m$ consists of $N_\mathup{ch}$ mutually disjoint ($\mathcal{R}^{\xi}_i\cap\mathcal{R}^{\xi}_j=\varnothing$ for $i\neq j$) subgraphs, each representing a high-contrast inclusion or channel. We denote the boundary of a subgraph $\mathcal{R}$ as
$$\partial\mathcal{R}\coloneqq\{x\in\mathcal{R}\colon\exists y\in \mathcal{N}\backslash\mathcal{R},\hspace{0.2em} \text{such that} \hspace{0.2em} x\sim y\}.$$
We assume
\begin{equation}
	\label{small contrast dependent}
	C_L^{\mathup{inf}}\xi\leq L_{xy} \leq C_L^{\mathup{sup}}\xi
\end{equation}
for all edges with $x,y\in\mathcal{R}^{\xi}$ and $y\sim x$, where $0<C_L^{\mathup{inf}}\leq C_L^{\mathup{sup}}$ are constants independent of $\xi$. We further assume that the sizes of the inclusions are independent of $\xi$ and that the diameters of $\mathcal{R}^{\xi}_m$ and $\mathcal{N}$ are of $O(1)$. If $\mathcal{N}$ embedded in $\mathbb{R}^2$ has a diameter of order $C_{\mathrm{po}}$, then the smallest contrast-independent eigenvalue scales as $O((C_{\mathrm{po}})^{-2})$. 

Consider the discrete generalized eigenvalue problem
$$A\psi_m=\lambda S\psi_m,$$
where
$$(Av,w)=\sum_{\mathcal{\mathcal{N}},\sim}L_{xy}(v(x)-v(y))(w(x)-w(y))+\sum_{x\in\mathcal{N}}M_xv(x) w(x),$$
$$(Sv,w)=\sum_{x\in\mathcal{N}}(\tilde{L}_{x}+M_x)v(x) w(x).$$
Here we also assume $0\leq M_x \leq\overline{M}<\infty$ as before and note that $\overline{M}$ is a constant independent of $\xi$.

In the following, we prove that $\lambda_m=O\big(\frac{1}{\xi}\big)$ for $m=1,2,...,N_{\mathrm{ch}}$ and $\lambda_m\geq C$ with $C$ independent of $\xi$ for $m\geq N_{\mathrm{ch}}+1$. 
First, we prove that there is at least $N_{\mathrm{ch}}$ small eigenvalues. Let $W\subset V$ be an arbitrary $N_{\mathrm{ch}}$-dimensional subspace. Note that we have
$$
\lambda_{N_{\mathrm{ch}}}=\min_{\dim(W)=N_{\mathrm{ch}}}\max_{v\in W \setminus \{0\}}R(v),\quad \text{where } R(v)=\frac{v^TAv}{v^TSv}.
$$
We need to find $W\subset V$ where the quotient $R(\cdot)$ is of order $1/\xi$. 
Define $V_\mathup{con}$ as the space spanned by vectors in $V$ that are constant inside each $\mathcal{R}^{\xi}_m$, $m=1,2,...,N_\mathup{ch}$. These functions are extended by zero outside $\mathcal{R}^{\xi}_m$. 
Let $v\in V_\mathup{con}$ and assume that $v=v_m$ in $\mathcal{R}^{\xi}_m$. Then
\begin{align*}
v^T A v&=\sum_{\mathcal{\mathcal{N}},\sim}L_{xy}(v(x)-v(y))^2+\sum_{x\in\mathcal{N}}M_xv(x)^2\\
&=\sum_{m=1}^{N_\mathup{ch}}\sum_{\substack{x\in\mathcal{R}^{\xi}_m \\ y\sim x}}\frac{1}{2}L_{xy}(v(x)-v(y))^2+\sum_{\substack{x\in\mathcal{R}_1\\ y\sim x}}\frac{1}{2}L_{xy}(v(x)-v(y))^2+\sum_{x\in\mathcal{N}}M_xv(x)^2\\
&=\sum_{m=1}^{N_\mathup{ch}}\sum_{\substack{x\in\partial\mathcal{R}^{\xi}_m \\ y\sim x,y\in\partial\mathcal{R}_1}}\frac{1}{2}(v(x)-v(y))^2+\sum_{m=1}^{N_\mathup{ch}}\sum_{\substack{x\in\partial\mathcal{R}_1\\ y\sim x,y\in\partial\mathcal{R}^\xi_m}}\frac{1}{2}(v(x)-v(y))^2+\sum_{m=1}^{N_\mathup{ch}}\sum_{x\in\mathcal{R}^\xi_m}M_xv_m^2\\
& =\sum_{m=1}^{N_\mathup{ch}}\sum_{\substack{x\in\partial\mathcal{R}^{\xi}_m \\ y\sim x,y\in\partial\mathcal{R}_1}} \frac{1}{2}(v(x)-v(y))^2+\sum_{m=1}^{N_\mathup{ch}}\sum_{x\in\mathcal{R}^\xi_m}M_xv_m^2\\
&\leq \tilde{C}\sum_{m=1}^{N_\mathup{ch}}\abs{\partial\mathcal{R}^{\xi}_m }v_m^2+\overline{M}\sum_{m=1}^{N_\mathup{ch}}\abs{\mathcal{R}^{\xi}_m }v_m^2 \leq (\tilde{C}+\overline{M})\sum_{m=1}^{N_\mathup{ch}}\abs{\mathcal{R}^{\xi}_m }v_m^2,
\end{align*}
where $\abs{\mathcal{R}^{\xi}_m},\abs{\partial\mathcal{R}^{\xi}_m}$ defined in the last two lines are the number of nodes in $\mathcal{R}^{\xi}_m,\partial\mathcal{R}^{\xi}_m$, respectively. $\tilde{C}$ is a regularity constant appearing in (2) of Assumption \ref{assumption on partitions}.
\begin{align*}
v^T S v&=\sum_{x\in\mathcal{N}}(\tilde{L}_{x}+M_x)v_m^2=\sum_{m=1}^{N_\mathup{ch}}\sum_{x\in\mathcal{R}^{\xi}_m}(\tilde{L}_x+M_x)v_m^2\\
&\geq \sum_{m=1}^{N_\mathup{ch}}\sum_{x\in\mathcal{R}^{\xi}_m}\big(\frac{1}{2}\sum_{\substack{y\sim x\\ y\in\mathcal{R}^{\xi}_m}}L_{xy}\big)v_m^2\geq \sum_{m=1}^{N_\mathup{ch}}\sum_{x\in\mathcal{R}^{\xi}_m}\big(\frac{1}{2}C_L^{\mathup{inf}}\xi\big)v_m^2\geq \frac{1}{2}C_L^{\mathup{inf}}\xi\sum_{m=1}^{N_\mathup{ch}}\abs{\mathcal{R}^{\xi}_m}v_m^2.
\end{align*}
Therefore, combining above two inequalities, we obtain that
\begin{align*}
R(v)=\frac{v^TAv}{v^TSv}\leq \frac{(\tilde{C}+\overline{M})\sum_{m=1}^{N_\mathup{ch}}\abs{\mathcal{R}^{\xi}_m }v_m^2}{\frac{1}{2}C_L^{\mathup{inf}}\xi\sum_{m=1}^{N_\mathup{ch}}\abs{\mathcal{R}^{\xi}_m}v_m^2}\leq \frac{2(\tilde{C}+\overline{M})}{C_L^{\mathup{inf}}}\cdot\frac{1}{\xi}.
\end{align*}
Now we present that there are at most $N_\mathup{ch}$ small eigenvalues. We have that
$$\lambda_{N_\mathup{ch}+1}=\min_{\dim(W)=N_\mathup{ch}+1}\max_{v\in W\backslash\{0\}}R(v).$$

Then we have to show that for every $(N_\mathup{ch}+1)$-dimensional subspace $W$, there exists a vector $v\in W$ such that $\mathcal{R}(v)\geq C'$, where $C'$ is independent of $\xi$.  Define the subspace 
$$V_{\mathup{po}}=\{v\in V\colon \frac{1}{2}\sum_{\substack{x\in\mathcal{R}^{\xi}_m\\ y\sim x}}v(x)=0\hspace{0.3em}{\rm for \hspace{0.3em} all}\hspace{0.3em}m=1,2,...,N_\mathup{ch}\}.$$
The subspace $V_\mathup{po}$ is of codimension $N_\mathup{ch}$. Note that for every $v\in V_\mathup{po}$, we can apply the poincare's inequality in every $\mathcal{R}^{\xi}_m$. Then we have
$$\sum_{\substack{x\in \mathcal{R}^\xi\\ y\sim x}}\frac{1}{2}v(x)^2=\sum_{m=1}^{N_\mathup{ch}}\frac{1}{2}\sum_{\substack{x\in\mathcal{R}^\xi_m\\ y\sim x}}v(x)^2\leq C_1\sum_{\substack{x\in \mathcal{R}^\xi\\ y\sim x}}\frac{1}{2}(v(x)-v(y))^2,$$
where $C_1$ is independent of $\xi$ , but depends on the number and volume of  $\{\mathcal{R}^\xi_m\}$. That is, we have
\begin{equation}
\label{poin 1}
(\xi-1)\sum_{\substack{x\in \mathcal{R}^\xi\\ y\sim x}}\frac{1}{2}v(x)^2\leq(\xi-1)C_1\sum_{\substack{x\in \mathcal{R}^\xi\\ y\sim x}}\frac{1}{2}(v(x)-v(y))^2
\end{equation}
for all $v\in V_\mathup{po}$. We can apply the standard poincare inequality to vectors on $V_\mathup{po}$
\begin{equation}
\label{poin 2}
\sum_{\mathcal{N},\sim}v(x)^2\leq C_2\sum_{\mathcal{N},\sim}(v(x)-v(y))^2,
\end{equation}
where $C_2$ is independent of $\xi$.
Adding (\ref{poin 1}) and (\ref{poin 2}), we obtain that there exists a constant $C_3$ independent of $\xi$ such that
\begin{align*}
\sum_{\mathcal{N},\sim}L_{xy}v(x)^2&\leq (\xi-1)\sum_{\substack{x\in \mathcal{R}^\xi\\ y\sim x}}\frac{1}{2}v(x)^2+\sum_{\mathcal{N},\sim}v(x)^2\leq C_3\sum_{\mathcal{N},\sim}L_{xy}(v(x)-v(y))^2.
\end{align*}
Therefore, we have
\begin{align*}
\sum_{\mathcal{N},\sim}L_{xy}v(x)^2+\sum_{x\in\mathcal{N}}M_xv(x)^2&\leq C^\#(\sum_{\mathcal{N},\sim}L_{xy}(v(x)-v(y))^2+\sum_{x\in\mathcal{N}}M_xv(x)^2),
\end{align*}
where $C^\#$ is independent of $\xi$, but depends on the number and volume of $\{\mathcal{R}^\xi_m\}$. 
Let $W\subset V$ be a subspace of dimension $N_\mathup{ch}+1$. We have that the intersection between $W$ and $V_\mathup{po}$ is a subspace of dimension at least one. Then we select $v\in W\cap V_\mathup{po}$ with $v\neq 0$, and for this vector, we get
$$R(v)=\frac{v^TAv}{v^TSv}=\frac{\sum_{\mathcal{N},\sim}L_{xy}(v(x)-v(y))^2+\sum_{x\in\mathcal{N}}M_xv(x)^2}{\sum_{x\in\mathcal{N}}(\tilde{L}_x+M_x)v(x)^2}\geq\frac{1}{C^\#},$$
where $C^{\#}$ does not depend on $\xi$ , but depends on the number and volume of $\{\mathcal{R}^\xi_m\}$.
Therefore, we obtain that $\lambda_m=O\big(\frac{1}{\xi}\big)$ for $m=1,2,...,N_{\mathrm{ch}}$ and $\lambda_m\geq C$ with $C$ independent of $\xi$ for $m\geq N_{\mathrm{ch}}+1$.

\section{An application to finite element discretization of elliptic PDEs} \label{application}
In this section, we apply our proposed method to finite element discretizations of elliptic partial differential equations. We show how the discretization in spatial networks retains similar properties to the multiscale method in the continuous setting.

We consider the domain $\Omega=(0,1)^2$ and the problem 
\begin{subequations}
	\label{model problem for the first example}
	\begin{align}
		-\text{div}(k\nabla u)&=g \quad \text{in}\quad \Omega, \label{model problem for the first example_a}\\
  u&=0\quad \text{on}\quad \partial\Omega,
	\end{align}
\end{subequations}
where $k$ is a high-contrast multiscale field. 
We define $\mathcal{T}_h$ be an unstructured fine mesh partition and 
$\mathcal{T}_H$ is coarse partition where each coarse element is connected and contains some complete fine meshes (see Figure \ref{fine mesh} where 23646 fine elements are partitioned into 100 parts.). 
We also denote $\mathcal{T}_H\coloneqq\{K_1,K_2,...,K_N\}$ (Recall that $N$ is the number of coarse elements). 

Let $\mathcal{N}_i$ be the set of all nodes contained in $K_i$ for $1\leq i\leq N$ and let $\mathcal{N}_i^l$ be the set of all nodes contained in $K^l_i$.
The weak form for (\ref{model problem for the first example}) is 
$$\int_{\Omega}k\nabla u\cdot\nabla vdx=\int_{\Omega}gvdx$$
for all $v\in H^1_0(\Omega)$. 

A standard fine finite element discretization (using continuous piecewise $P1$ polynomial) for this problem is: find $u_h\in V_h\coloneqq\text{span}\{\beta_1,\beta_2,...,\beta_{N_h}\}$ such that (here $N_h\coloneqq \#{\mathcal{N}}$ is the number of nodes contained in $\mathcal{N}$. 
$\{\beta_i\}_{i=1}^{N_h}$ are standard nodal basis functions corresponding to the nodes $\mathcal{N}=\{x_1,x_2,...,x_{N_h}\}$ defined on the fine mesh.)
\begin{equation}
\label{P1 discretization for first example}
\int_{\Omega}k\nabla u_h\cdot\nabla v_hdx=\int_{\Omega} gv_hdx\quad\forall v_h\in V_h.
\end{equation}
The underlying matrix form for (\ref{P1 discretization for first example}) can be expressed as 
\begin{equation}
\label{matrix form for weak form}
A\mathbf{u}=\mathbf{b},
\end{equation}
where $A_{ij}=\int_{\Omega}k\nabla\beta_j\cdot\nabla\beta_idx$, $\mathbf{u}=(u_1,u_2,...,u_{N_h})^T$, $\mathbf{b}=(b_1,b_2,...,b_{N_h})^T$ and $b_i=\int_{\Omega}g\beta_idx$. 
It is also clear that $u_i=u_h(x_i)$. 

Next we consider the relationship between $A$ and the $(L+M)$-form matrix defined in our method. Recall that the vector space $V$ is the space of all real-valued functions defined on the nodes $\mathcal{N}$.
Since $V$ is an isomorphism to $\Real^{N_h}$, we can take a basis for $V\coloneqq \{\mathbf{e}_1,\mathbf{e}_2,...,\mathbf{e}_{N_h}\}$, where the $i$-th component of $\mathbf{e}_i$ is 1 (i.e. $\mathbf{e}_i(x_i)=1$) and other components of $\mathbf{e}_i$ are zeros. 
In term of (\ref{model problem}),
$$\sum_{\mathcal{N},\sim}L_{xy}(u(x)-u(y))(\mathbf{e}_i(x)-\mathbf{e}_i(y))+\sum_{x\in\mathcal{N}}M_xu(x)\mathbf{e}_i(x)=\sum_{x\in\mathcal{N}}f(x)\mathbf{e}_i(x),$$
for all $i=1,2,..,N_h$. 
Since $\mathbf{e}_i(x_j)=\delta_{ij}$, we have the following linear system:
\begin{subequations}
	\label{linear equations for our method}
	\begin{align}
		\sum_{\substack{x_1,y\in\mathcal{N}\\ y\sim x_1}}\frac{1}{2}L_{x_1 y}(u(x_1)-u(y))+M_{x_1}u(x_1)&=f(x_1),\\
  \sum_{\substack{x_2,y\in\mathcal{N}\\ y\sim x_2}}\frac{1}{2}L_{x_2 y}(u(x_2)-u(y))+M_{x_2}u(x_2)&=f(x_2),\\
  ....\nonumber&\\
  \sum_{\substack{x_{N_h},y\in\mathcal{N}\\ y\sim x_{N_h}}}\frac{1}{2}L_{x_{N_h}y}(u(x_{N_h})-u(y))+M_{x_{N_h}}u(x_{N_h})&=f(x_{N_h}).
	\end{align}
\end{subequations}
From (\ref{linear equations for our method}), we clearly find the matrix $L,M$ and the vecter $f$. 
More precisely, $L_{ii}=\sum_{y\sim x_i}L_{x_iy}$, $L_{ij}=-L_{x_ix_j}$ for $x_i\sim x_j$ and $L_{ij}=0$ if $x_i$ and $x_j$ are not connected. 
$M=\text{diag}(M_{x_1},M_{x_2},...,M_{x_{N_h}})$. $f=(f_1,f_2,...,f_{N_h})^T$ with $f_i=f(x_i)$. 
Therefore, when we apply our method to the first example, we find the relationship between the finite element discretization for the first example and our method is:
$$A=L+M,\quad A_{ij}=L_{ij}+M_{ij},\quad b_i=f_i$$
for $1\leq i,j\leq N_h$. 
The solution $\mathbf{u}=(u_1,u_2,...,u_{N_h})^T$ with $u_i=u_h(x_i)$ obtained in (\ref{matrix form for weak form}) plays the same role as the solution $(u(x_1),u(x_2),...,u(x_{N_h}))^T$ computed in (\ref{linear equations for our method}).

Next we apply our method to construct the multiscale basis functions in the following.
We first construct the auxiliary multiscale basis space.
The local spectral problem is: find a real number $\lambda^i_{j}$ and a function $\phi^i_{j}\in V_h({K}_i)$ such that
\begin{equation}
\label{local spectral problem for the first example}
\int_{K_i}k\nabla\phi^i_{j}\cdot\nabla w_hdx=\lambda^i_j\int_{K_i}kH^{-2}\phi^i_{j}\cdot w_hdx\quad \forall w_h\in V_h({K}_i),
\end{equation}
where $V_h({K}_i)$ is the restriction of $V_h$ on $K_i$. Let $\beta_{i,l}$ $(1\leq l\leq N_{h,i})$ is nodal basis function corresponding to nodes in $\mathcal{N}_i$,
and $N_{h,i}\coloneqq\#{{\mathcal{N}}_i}$ is the number of nodes contained in $\mathcal{N}_i$. 
We denote that the nodal basis functions $\{\beta_{i,1},\beta_{i,2},...,\beta_{i,N_{h,i}}\}$ correspond to $\{x_{i,1},x_{i,2},...,x_{i,N_{h,i}}\}$ which are the nodes contained in $\mathcal{N}_i$.
Then the matrix representation for (\ref{local spectral problem for the first example}) is
\begin{equation}
\label{local spectral matrix form}
A_{i}\mathbf{c}_{i}=\lambda^i_jB_{i}\mathbf{c}_{i},
\end{equation}
where $(A_i)_{lj}=\int_{K_i}k\nabla\beta_{i,j}\cdot\nabla\beta_{i,l}$, $(B_i)_{lj}=\int_{K_i}kH^{-2}\beta_{i,j}\beta_{i,l}$, and $\mathbf{c}_{i}=(c_1,c_2,...,c_{N_{h,i}})^T$ with $c_l=\phi^i_{j}(x_{i,l})$ (Actually, we let $\phi^i_{j}=\sum_{l=1}^{N_{h,i}}c_l\beta_{i,l}$ in the local spectral problem (\ref{local spectral problem for the first example})) for $1\leq l,j\leq N_{h,i}$. 

Next we consider the relationship between (\ref{local spectral matrix form}) and the $(L+M)$-form matrix restricted to $\mathcal{N}_i$ defined in our method. Recall that the vector space $V(\mathcal{N}_i)$ is the space of all real-valued functions defined on the nodes in $\mathcal{N}_i$.
Since $V(\mathcal{N}_i)$ is an isomorphism to $\Real^{N_{h,i}}$, we can take a basis for $V(\mathcal{N}_i)$: $\{\mathbf{e}_{i,1},\mathbf{e}_{i,2},...,\mathbf{e}_{i,N_{h,i}}\}$, where the $j$-th component of  $\mathbf{e}_{i,l}$ is  $\mathbf{e}_{i,l}(x_{i,j})=\delta_{lj}$. 

In terms of (\ref{spectral problem})-(\ref{def of a and s}), we have
$$\sum_{\mathcal{N}_i,\sim}L_{xy}(\phi^i_j(x)-\phi^i_j(y))(\mathbf{e}_{i}(x)-\mathbf{e}_{i}(y))+\sum_{x\in\mathcal{N}_i}M_x\phi^i_j(x)\mathbf{e}_{i}(x)=\lambda^i_j\sum_{x\in\mathcal{N}_i}\phi^i_j(x)\mathbf{e}_{i}(x),$$
for all $i=1,2,..,N_{h,i}$. 
Since $\mathbf{e}_{i}(x_{i,j})=\delta_{ij}$, we have the following linear system on each $\mathcal{N}_i(1\leq i\leq N)$:
\begin{subequations}
	\label{linear equations for local spectral problem}
	\begin{align}
		\sum_{\substack{y\in\mathcal{N}_i\\ y\sim x_{i,1}}}\frac{1}{2}L_{x_{i,1}y}(\phi^i_j(x_{i,1})-\phi^i_j(y))+M_{x_{i,1}}\phi^i_j(x_{i,1})&=\lambda^i_j(\tilde{L}_{x_{i,1}}+M_{x_{i,1}})(C_{\mathrm{po}})^{-2}\cdot\phi^i_j(x_{i,1}),\\
  \sum_{\substack{y\in\mathcal{N}_i\\ y\sim x_{i,2}}}\frac{1}{2}L_{x_{i,2}y}(\phi^i_j(x_{i,2})-\phi^i_j(y))+M_{x_{i,2}}\phi^i_j(x_{i,2})&=\lambda^i_j(\tilde{L}_{x_{i,2}}+M_{x_{i,2}})(C_{\mathrm{po}})^{-2}\cdot\phi^i_j(x_{i,2}),\\
  ....\nonumber&\\
  \sum_{\substack{y\in\mathcal{N}_i\\ y\sim x_{i,N_{h,i}}}}\frac{1}{2}L_{x_{i,N_{h,i}}y}(\phi^i_j(x_{i,N_{h,i}})-\phi^i_j(y))+M_{x_{i,N_{h,i}}}\phi^i_j(x_{i,N_{h,i}})
&=\lambda^i_j(\tilde{L}_{x_{i,N_{h,i}}}+M_{x_{i,N_{h,i}}})(C_{\mathrm{po}})^{-2}\cdot\phi^i_j(x_{i,N_{h,i}}).
\end{align}
\end{subequations}
From (\ref{linear equations for local spectral problem}), we clearly represent the local spectral problem in the matrix form as
\begin{equation}
\label{local matrix}
(L_i+M_i)\mathbf{p}_{i}=\lambda^i_jS_i\mathbf{p}_{i},
\end{equation}
where $(L_{i})_{ll}=\sum_{\substack{y\in\mathcal{N}_i\\ y\sim x_{i,l}}}L_{x_{i,l}y}$, $(L_i)_{lm}=-L_{x_{i,l}x_{i,m}}$ for $x_{i,l}\sim x_{i,m}$ and $(L_i)_{lm}=0$ if $x_{i,l}$ and $x_{i,m}$ are not connected. 
$M_i=\text{diag}(M_{x_{i,1}},M_{x_{i,2}},...,M_{x_{i,N_h}})$. $(S_i)_{lm}=\delta_{lm}(C_{\mathrm{po}})^{-2}(\tilde{L}_{x_{i,l}}+M_{x_{i,l}})$. 
$\mathbf{p}_{i}=(p_1,p_2,...,p_{N_{h,i}})^T$ with $p_l=\phi^i_j(x_{i,l})$ (Actually, we have assumed $\phi^i_j=\sum_{l=1}^{N_{h,i}}p_l\mathbf{e}_{i,l}$ in (\ref{linear equations for local spectral problem})). 

Note that the solution $(\lambda^i_j,\mathbf{c}_{i})$  in (\ref{local spectral matrix form}) plays the same role as the solution $(\lambda^i_j,\mathbf{p}_{i})$ computed in (\ref{local matrix}).

 Then we use the first $l_i$ eigenfunctions to construct our local auxiliary multiscale space $V_\mathup{aux}^i$ and global auxiliary multiscale space $V_\mathup{aux}$. 
 Based on the auxiliary multiscale space,  we can construct the multiscale basis functions. 
 
 For easier understanding, we will also describe the relationship between the finite element discretization for the first example and $(L+M)$-form matrix in our method. 
 That is, we talk about how our method applied to the finite element discretization of the first example. 
 We find $\psi^i_{j,l}\in V_{h}(K_i^l)$ ($V_{h}(K_i^l)$ is the restriction of $V_h$ on $K_i^l$ ) such that
 \begin{equation}
\label{minin-variational problem}
\int_{K_i^l}k\nabla\psi^i_{j,l}\cdot\nabla w_hdx+\int_{K_i^l}\pi(\psi^i_{j,l})\cdot\pi(w_h)dx=\int_{K_i^l}kH^{-2}\phi^i_j\cdot w_hdx\quad \forall w_h\in V_{h}(K_i^l)
 \end{equation}
In fact, we can take $V_{h}(K_i^l)\coloneqq\text{span}\{\beta^l_{i,1},\beta^l_{i,2},...,\beta^l_{i,N^l_i}\}$, where $N^l_i\coloneqq\# \mathcal{N}^l_i$ is the number of nodes contained in $\mathcal{N}^l_i$. $\{\beta^l_{i,1},\beta^l_{i,2},...,\beta^l_{i,N^l_i}\}$ are nodal basis functions corresponding to the nodes $\{X^l_{i,1},X^l_{i,2},...,X^l_{i,N^l_i}\}$ contained in $\mathcal{N}^l_i$. Then the matrix representation for (\ref{minin-variational problem}) is
\begin{equation}
\label{oversampling matrix form}
(A^l_i+S^l_i)\mathbf{c}^l_i=\mathbf{b}^l_i,
\end{equation}
where $(A^l_i)_{lj}=\int_{K_i^l}k\nabla\beta^l_{i,j}\cdot\nabla\beta^l_{i,l}dx$, $(S^l_i)_{lj}=\int_{K_i^l}kH^{-2}\pi(\beta^l_{i,j})\cdot\pi(\beta^l_{i,l})dx$, and $\mathbf{b}^l_i=(b_1,b_2,...,b_{N^l_i})^T$ with $b_j=\int_{K_i^l}kH^{-2}\phi^i_j\cdot\beta^m_{i,l}dx$.  $\mathbf{c}^l_i=(c_1,c_2,...,c_{N^l_i})^T$ with $c_j=\psi^i_{j,l}(X^l_{i,j})$. (Actually, we let $\psi^i_{j,l}=\sum_{m=1}^{N^l_i}c_l\beta^l_{i,m}$ in the variational form (\ref{minin-variational problem}).) 

Next we consider the relationship between (\ref{oversampling matrix form}) and the $(L+M)$-form matrix restricted to $\mathcal{N}^l_i$ defined in our method. Recall that the vector space $V(\mathcal{N}^l_i)$ is the space of all real-valued functions defined on the nodes $\mathcal{N}^l_i$.
Since $V(\mathcal{N}^l_i)$ is an isomorphism to $\Real^{\#\mathcal{N}^l_i}$, we can take a basis for $V(\mathcal{N}^l_i)$: $\{\mathbf{e}^l_{i,1},\mathbf{e}^l_{i,2},...,\mathbf{e}^l_{i,N^l_i}\}$, where the $l$-th component of  $\mathbf{e}^l_{i,j}$ is  $\mathbf{e}^l_{i,j}(X^l_{i,m})=\delta_{jm}$. 
In terms of (\ref{variational_relaxed_cem}),
\begin{align}
\label{oversampling L M}
&\quad\sum_{\mathcal{N},\sim}\frac{1}{2}L_{xy}(\psi^i_{j,l}(x)-\psi^i_{j,l}(y))(v(x)-v(y))+\sum_{x\in\mathcal{N}}M_x\psi^i_{j,l}(x)v(x)+\sum_{x\in\mathcal{N}}(\tilde{L}_x+M_x)(C_{\mathrm{po}})^{-2}\cdot\psi^i_{j,l}(x)\cdot\pi v(x)\nonumber\\
&=\sum_{x\in\mathcal{N}}(\tilde{L}_x+M_x)(C_{\mathrm{po}})^{-2}\cdot\phi^i_j(x)\cdot v(x)
\end{align}
for all $v\in V(\mathcal{N}^l_i)$. 
Note that 
\begin{align*}
(\pi(\mathbf{e}^l_{i,k}))(X^l_{i,m})&=(\sum_{i=1}^N\sum_{j=1}^{l_i}s_i(\mathbf{e}^l_{i,k},\phi^i_j)\cdot\phi^i_j)(X^l_{i,m}) =\big(\sum_{i=1}^N\sum_{j=1}^{l_i}\phi^i_j\cdot(\sum_{x\in\mathcal{N}_i}(\tilde{L}_x+M_x)(C_{\mathrm{po}})^{-2}\mathbf{e}^l_{i,k}(x)\cdot\phi^i_j(x))\big)(X^l_{i,m})\\
&=\sum_{j=1}^{l_i}\delta_{km}(\tilde{L}_{X^l_{i,m}}+M_{X^l_{i,m}})(C_{\mathrm{po}})^{-2}(\phi^i_j(X^l_{i,m}))^2.
\end{align*}
Taking test function $v=\mathbf{e}^l_{i,j}$ for $1\leq j\leq N^l_i$ in (\ref{oversampling L M}), we have the following linear system
\begin{subequations}
	\label{linear equations for relax cem}
\begin{align}
&\sum_{y\sim X^l_{i,1}}L_{X^l_{i,1}y}(\psi^i_{j,l}(X^l_{i,1})-\psi^i_{j,l}(y))+M_{X^l_{i,1}}\psi^i_{j,l}(X^l_{i,1})\nonumber\\
&+(\tilde{L}_{X^l_{i,1}}+M_{X^l_{i,1}})^2(C_{\mathrm{po}})^{-2}(\sum_{j=1}^{l_i}(\phi^i_j(X^l_{i,1}))^2)\psi^i_{j,l}(X^l_{i,1})=(\tilde{L}_{X^l_{i,1}}+M_{X^l_{i,1}})(C_{\mathrm{po}})^{-2}\cdot\phi^i_{j}(X^l_{i,1}),\\
 &\sum_{y\sim X^l_{i,2}}L_{X^l_{i,2}y}(\psi^i_{j,l}(X^l_{i,2})-\psi^i_{j,l}(y))+M_{X^l_{i,2}}\psi^i_{j,l}(X^l_{i,2})\nonumber\\
  &+(\tilde{L}_{X^l_{i,2}}+M_{X^l_{i,2}})^2(C_{\mathrm{po}})^{-2}(\sum_{j=1}^{l_i}(\phi^i_j(X^l_{i,2}))^2)\psi^i_{j,l}(X^l_{i,2})=(\tilde{L}_{X^l_{i,2}}+M_{X^l_{i,2}})(C_{\mathrm{po}})^{-2}\cdot\phi^i_{j}(X^l_{i,2}),\\
  &\quad\quad......\nonumber\\
  &\sum_{y\sim X^l_{i,N^l_i}}L_{X^l_{i,N^l_i}y}(\psi^i_{j,l}(X^l_{i,N^l_i})-\psi^i_{j,l}(y))+M_{X^l_{i,N^l_i}}\psi^i_{j,l}(X^l_{i,N^l_i})\nonumber\\
  &+(\tilde{L}_{X^l_{i,N^l_i}}+M_{X^l_{i,N^l_i}})^2(C_{\mathrm{po}})^{-2}(\sum_{j=1}^{l_i}(\phi^i_j(X^l_{i,N^l_i}))^2)\psi^i_{j,l}(X^l_{i,N^l_i})=(\tilde{L}_{X^l_{i,N^l_i}}+M_{X^l_{i,N^l_i}})(C_{\mathrm{po}})^{-2}\cdot\phi^i_{j}(X^l_{i,N^l_i}).
\end{align}
\end{subequations}
The corresponding matrix form for (\ref{linear equations for relax cem}) is
\begin{equation}
\label{matric for oversampling L M}
(L^l_i+M^l_i+R^l_i)\mathbf{p}^l_i=\mathbf{q}^l_i,
\end{equation}
where $(L^l_{i})_{ll}=\sum_{y\sim X^l_{i,m}}L_{X^l_{i,m}y}$, $(L^l_i)_{mj}=-L_{X^l_{i,m}X^l_{i,j}}$ for $X^l_{i,m}\sim X^l_{i,j}$ and $(L^l_i)_{mj}=0$ if $X^l_{i,m}$ and $X^l_{i,j}$ are not connected. 
$M^l_i=\text{diag}(M_{X^l_{i,1}},M_{X^l_{i,2}},...,M_{X^l_{i,N^l_i}})$. $(R^l_i)_{mj}=\delta_{mj}(C_{\mathrm{po}})^{-2}(\tilde{L}_{X^l_{i,m}}+M_{X^l_{i,m}})^2(\sum_{j=1}^{l_i}(\phi^i_j(X^l_{i,N^l_i}))^2)$. $\mathbf{q}^l_i=(q_1,q_2,...,q_{N^l_i})^T$ with $q_l=(\tilde{L}_{X^l_{i,m}}+M_{X^l_{i,m}})(C_{\mathrm{po}})^{-2}\cdot\phi^i_{j}(X^l_{i,m})$.  
$\mathbf{p}^l_i=(p_1,p_2,...,p_{N^l_i})^T$ with $p_l=\psi^i_{j,l}(X^l_{i,m})$ (Actually, we have assumed $\psi^i_{j,l}=\sum_{m=1}^{N^l_i}p_m\mathbf{e}^l_{i,m}$ in (\ref{variational_relaxed_cem})). 

Note that the solution $\mathbf{c}^l_{i}$  in (\ref{oversampling matrix form}) plays the same role as the solution $\mathbf{p}^l_i$ computed in (\ref{matric for oversampling L M}).
More precisely, $A^l_{i}$ corresponds to $L^l_i+M^l_i$, $S^l_{i}$ corresponds to $R^l_i$ and $\mathbf{b}^l_{i}$ corresponds to $\mathbf{q}^l_i$. 

After obtaining the multiscale basis functions and global multiscale basis functions, we can solve corresponding linear systems to obtain $u_\mathup{glo}$ and $u_\mathup{ms}$. 

\section{Numerical experiments}
\label{numerical results}

In this section, we present a series of comprehensive numerical experiments to demonstrate the efficiency, accuracy, and robustness of the proposed CEM-GMsFEM framework for spatial network models. To ensure the reliability and reproducibility of our results, all numerical simulations and evaluations are conducted on a high-performance computing workstation. The hardware configuration features dual Intel\textsuperscript{\textregistered} Xeon\textsuperscript{\textregistered} Gold 6230 processors, providing a total of 80 computational cores, along with 1024 GB of main memory (RAM). Furthermore, to promote open science and facilitate community engagement, the complete source code and implementation details used in this study have been made publicly accessible via our GitHub repository\footnote{\href{https://github.com/pentaery/Algebraic\_CEM}{https://github.com/pentaery/Algebraic\_CEM}}.

The implementation workflow is as follows. We first generate the finite element mesh using Gmsh \cite{geuzaine2009gmsh} and assemble two sparse matrices, L and U from \Cref{Mass} and \Cref{def_L}. We then apply METIS \cite{karypis1997metis} to the graph induced by $L$ and partition it into $\mathcal{N}$ non-overlapping subregions, which are used to construct the coarse decomposition and oversampling neighborhoods. Finally, all major linear algebra operations in the offline and online stages are performed with Intel MKL routines \cite{wang2014intel} (including sparse matrix kernels and linear solvers), ensuring computational efficiency and scalability.

Since the number of subregions $\mathcal{N}$ is typically large, the constructions of $V_\mathup{aux}$ and $V_\mathup{ms}$ involve many repeated local loops. These local computations are largely independent across subregions (and basis indices), so we use OpenMP to parallelize the corresponding loops. In practice, each thread handles a subset of local spectral problems and local energy-minimization solves, while synchronization is only required at the assembly stage. This parallel strategy substantially reduces the offline computational time without changing the numerical results.
To quantify this acceleration, we provide a dedicated strong-scaling and weak-scaling evaluation protocol in \Cref{sec:openmp-performance}.

\subsection{Poisson equation with heterogeneous coefficients}
To systematically evaluate the performance of our proposed multiscale method, we consider the classical Poisson equation with highly heterogeneous coefficients and homogeneous Dirichlet boundary conditions. The underlying setting adheres to the finite element discretization framework presented in Section \ref{application}. 

We conduct our numerical experiments on a representative case with a regular square mesh partition with complex highly heterogeneous coefficients as shown in Figure \ref{fig:mesh_partition}.  The domain is discretized using conforming piecewise linear ($P_1$) finite elements generated via Gmsh. Following the standard Galerkin assembly process, this fine-scale discretization naturally yields a large-scale linear system $A\mathbf{u}=\mathbf{b}$, which serves as the direct algebraic input for our multiscale basis construction. Unless otherwise specified, the right-hand side source term is chosen as 
\[
\mathbf{b} = 2\pi^2\sin(\pi x)\sin(\pi y)
\] 
in all subsequent numerical examples.

A core component of the proposed method is the localized formulation based on overlapping subgraphs. First, the fine-scale matrix connectivity graph is partitioned into 100 non-overlapping subdomains using the METIS algorithm (see Figure \ref{fig:mesh_partition}). Subsequently, overlapping regions are systematically generated by expanding each localized subdomain layer by layer. The appropriate construction of these oversampled spaces is crucial for guaranteeing the exponential decay of the multiscale basis functions. Figure \ref{fig:overlap} visually illustrates the progressive expansion of an arbitrary subgraph into 1-layer, 2-layer, 3-layer, and 4-layer overlapping regions.

\begin{figure}[ht!]
    \centering
    \begin{subfigure}[b]{0.49\textwidth}
        \centering
        \includegraphics[width=\textwidth]{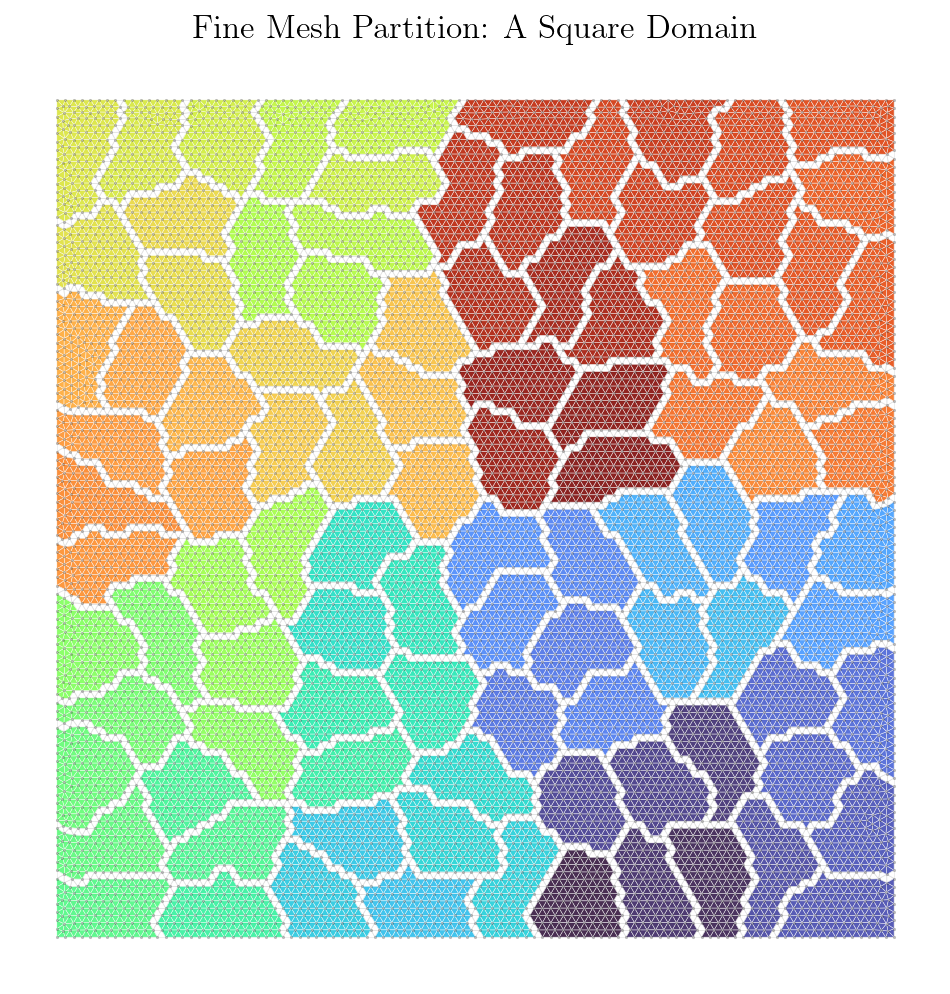}
        \caption{}
        \label{fig:fine mesh}
    \end{subfigure}
    \hspace{0.5em}%
    \begin{subfigure}[b]{0.49\textwidth}
        \centering
        \raisebox{0.6cm}{\includegraphics[width=\textwidth]{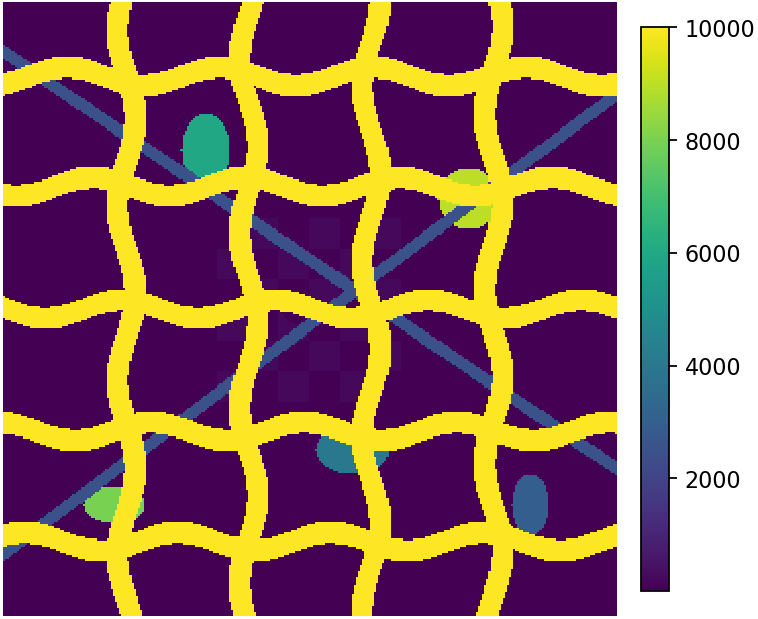}}
        \caption{}
        \label{fig:coarse element}
    \end{subfigure}
    \caption{(a) Square domain discretized into 23,646 elements, partitioned into 100 subgraphs; (b) The high-contrast medium configuration.}
    \label{fig:mesh_partition}
\end{figure}

\begin{figure}[ht!]
    \centering
    \begin{subfigure}{0.23\textwidth}
        \centering
        \includegraphics[width=\textwidth]{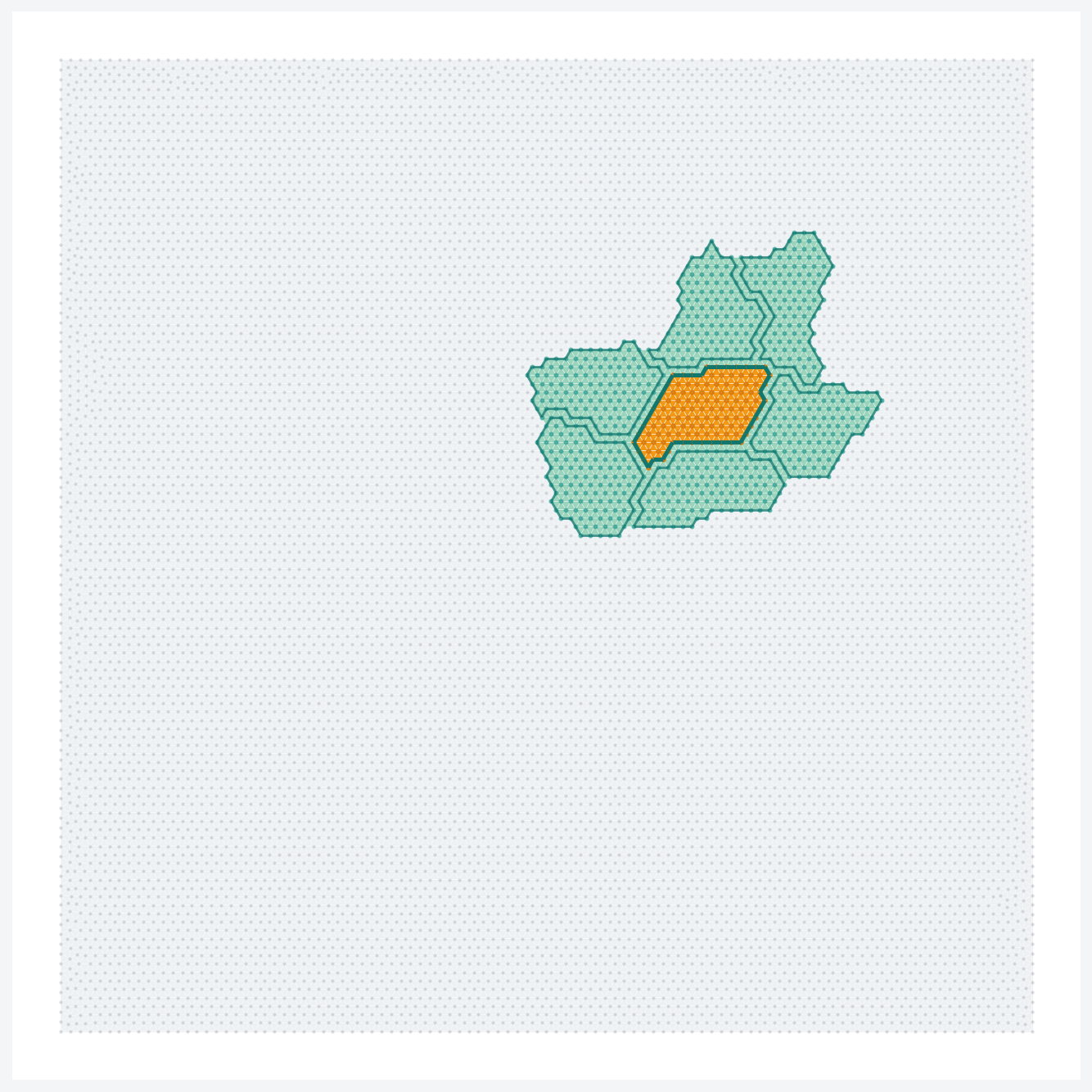}
        \label{overlap_1}
    \end{subfigure}
    \hspace{0.5em}%
    \begin{subfigure}{0.23\textwidth}
        \centering
        \includegraphics[width=\textwidth]{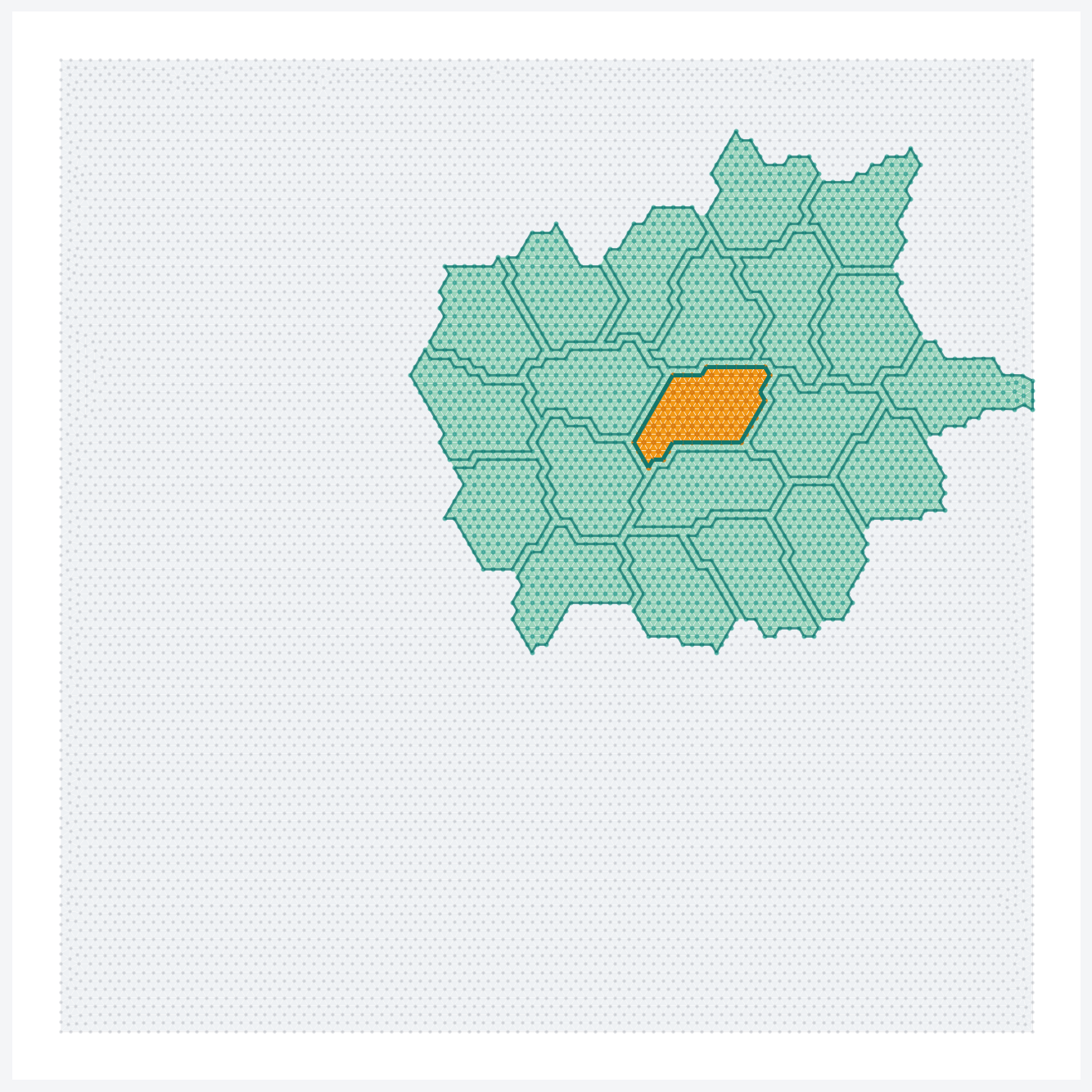}
        \label{overlap_2}
    \end{subfigure}
        \hspace{0.5em}%
    \begin{subfigure}{0.23\textwidth}
        \centering
        \includegraphics[width=\textwidth]{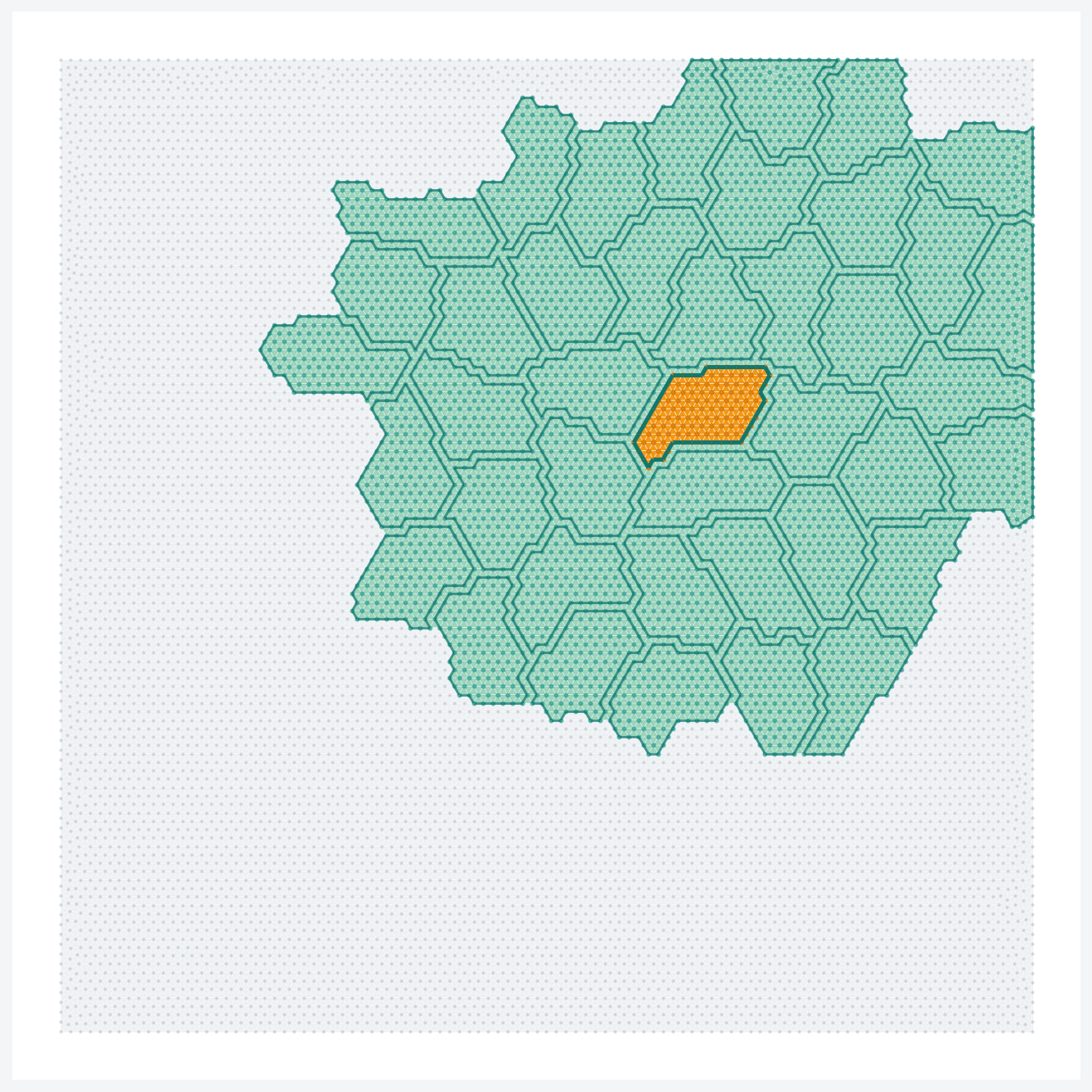}
        \label{overlap_3}
    \end{subfigure}
        \hspace{0.5em}%
    \begin{subfigure}{0.23\textwidth}
        \centering
        \includegraphics[width=\textwidth]{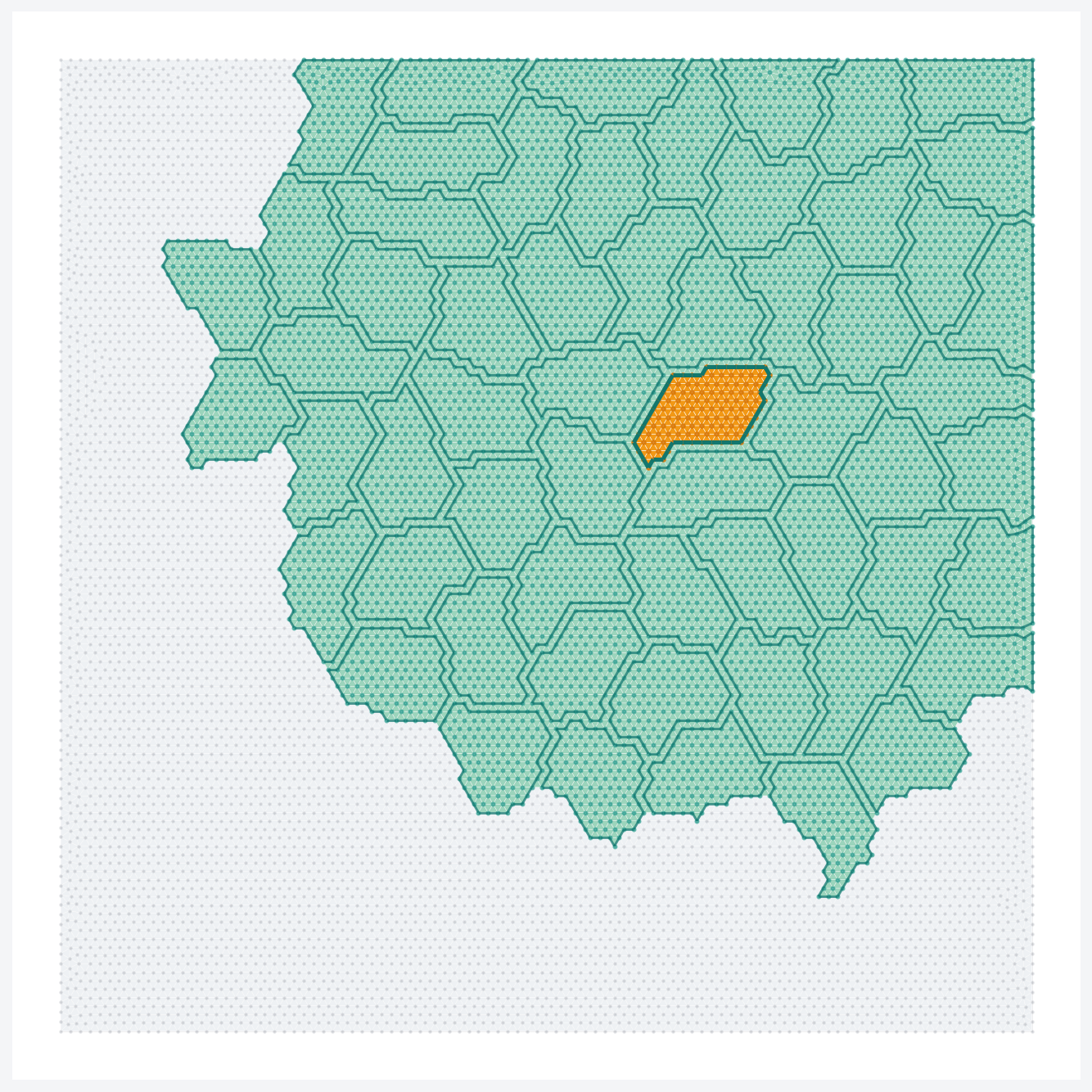}
        \label{overlap_4}
    \end{subfigure}
    \caption{Illustration of an arbitrary base subgraph and its corresponding oversampled regions, demonstrating 1-layer, 2-layer, 3-layer, and 4-layer overlaps (from left to right).}
    \label{fig:overlap}
\end{figure}

To quantitatively evaluate the approximation quality of the multiscale solver, we take the fine-scale direct solution $u_h$ as the reference and define the following relative $L^2$ and energy errors for a given number of overlapping layers $l$:
\[
e_{L^2}^m:=\frac{\norm{u_h-u_\mathup{cem}^l}_{L^2(\Omega)}}{\norm{u_h}_{L^2(\Omega)}},
\qquad
e_{a}^m:=\frac{\norm{u_h-u_\mathup{cem}^l}_{a}}{\norm{u_h}_{a}}.
\]
Here the energy norm is defined by
\[
\norm{v}_{a}^{2}=\int_{\Omega}k\nabla v\cdot\nabla v\,dx,
\]
and can be equivalently expressed in the algebraic matrix form as:
\[
\norm{\mathbf{v}}_{L^2(\Omega)}^{2}=\mathbf{v}^{T}M_h\mathbf{v},
\qquad
\norm{\mathbf{v}}_{a}^{2}=\mathbf{v}^{T}A\mathbf{v},
\]
where $M_h$ represents the fine-scale mass matrix.

We first examine the convergence behavior of the CEM-GMsFEM solution with respect to the number of oversampling layers $l$ and the number of partitioned subgraphs $n$. Here we use Case 1 with 290,126 vertices and $l=4$.
\begin{table}[ht!]
    \centering
    \caption{Relative $L^2$ and energy errors for different $n$ and $l$.}
    \label{tab:comparenm}
    \begin{tabular}{ccccc}
        \toprule
        n & 500 & 1000 & 2000 & 4000 \\
        \midrule
        
         $e_{L^2}^2$ &3.12\% &3.19\% &4.86\% & 7.15\%\\
        $e_{L^2}^3$ &2.34\% &1.24\% &1.73\% & 2.69\%\\
         $e_{L^2}^4$ &2.24\% &1.02\% &0.79\% & 1.10\%\\
         $e_{L^2}^5$ &2.22\% &0.97\% &0.55\% & 0.55\%\\
        \midrule
        
         $e_{a}^2$ &20.37\% &23.65\% &29.74\% & 35.65\%\\
         $e_{a}^3$ &13.21\% &13.30\% &17.43\% & 22.10\%\\
         $e_{a}^4$ &11.10\% &8.87\% &10.83\% & 13.88\% \\
         $e_{a}^5$ &10.46\% &7.13\% &7.45\% & 9.28\% \\
        \bottomrule
    \end{tabular}
\end{table}

Table \ref{tab:comparenm} reveals a clear coupling between the partition number $n$ and the oversampling layers $l$. For fixed $n$, increasing $l$ consistently reduces both $L^2$ and energy errors, indicating effective decay of localization error as the oversampling region expands. For example, at $n=4000$, the $L^2$ error decreases from $7.15\%$ ($l=2$) to $0.55\%$ ($l=5$), while the energy error decreases from $35.65\%$ to $9.28\%$. In contrast, when $l$ is small ($l=2,3$), increasing $n$ may deteriorate accuracy, which suggests that finer partitions require sufficiently large oversampling to capture cross-subgraph interactions. Once $l\ge 4$, this adverse trend is largely alleviated and finer partitions become competitive; the best $L^2$ error is $0.55\%$ (at $n=2000$ or $4000$, $l=5$), while the best energy error is $7.13\%$ (at $n=1000$, $l=5$).

Next we fix the number of subgraphs to $n=2000$ and examine the impact of the overlap layers $l$, the number of eigenvectors $Nov$, and the contrast $\kappa^*$ on the accuracy of the CEM solution.

Table \ref{tab:impact-lkm} demonstrates the robustness and accuracy of the CEM solution concerning the oversampling layers ($l$), the number of eigenvectors per subgraph ($Nov$), and the contrast ratio ($\kappa^*$). Several key observations can be drawn from the results:
First, for a fixed contrast and $Nov$, enlarging the oversampling region (i.e., increasing $l$ from 3 to 5) consistently and significantly reduces both the $L^2$ and energy errors, which confirms the exponential decay of the localization errors. 
Second, enriching the local multiscale space by increasing the number of eigenvectors $Nov$ effectively improves the accuracy. For instance, at $\kappa^*=10^5$ and $l=5$, the energy error decreases from $8.40\%$ to $6.21\%$ as $Nov$ increases from 3 to 6. 
Finally, the method exhibits excellent robustness with respect to the high contrast $\kappa^*$. As the contrast increases from $10^5$ to $10^6$, the relative errors do not deteriorate (and the energy errors actually decrease notably in all considered cases), indicating that the constructed multiscale basis functions successfully capture the dominant heterogeneous features of the media across varying levels of contrast.

\begin{table}[htbp]
    \centering
    \caption{Relative $L^2$ and Energy errors of the CEM solution for different $l$, $Nov$, and $\kappa^*$}
    \label{tab:impact-lkm}
    \begin{tabular}{cccccccc}
        \toprule
        & & \multicolumn{2}{c}{$l=3$} & \multicolumn{2}{c}{$l=4$} & \multicolumn{2}{c}{$l=5$} \\
        \cmidrule(lr){3-4} \cmidrule(lr){5-6} \cmidrule(lr){7-8}
        Contrast $\kappa^*$ & $Nov$ & $e_{L^2}^3$ & $e_{a}^3$ & $e_{L^2}^4$ & $e_{a}^4$ & $e_{L^2}^5$ & $e_{a}^5$ \\

        \midrule
        \multirow{4}{*}{$10^5$} & $3$ & 2.06\% & 18.61\% & 1.12\% & 11.89\% & 0.89\% & 8.40\% \\
        & $4$ & 1.73\% & 17.43\% & 0.79\% & 10.83\% & 0.55\% & 7.45\% \\
        & $5$ & 1.53\% & 16.40\% & 0.67\% & 10.04\% & 0.45\% & 6.80\% \\
        & $6$ & 1.34\% & 15.31\% & 0.58\% & 9.26\% & 0.40\% & 6.21\% \\
        \midrule
        \multirow{4}{*}{$10^6$} & $3$ & 1.14\% & 9.45\% & 0.94\% & 6.77\% & 0.90\% & 5.82\% \\
        & $4$ & 0.89\% & 8.46\% & 0.70\% & 5.77\% & 0.67\% & 4.81\% \\
        & $5$ & 0.74\% & 7.72\% & 0.56\% & 5.03\% & 0.54\% & 4.04\% \\
        & $6$ & 0.59\% & 7.02\% & 0.41\% & 4.43\% & 0.39\% & 3.46\% \\
        \bottomrule
    \end{tabular}
\end{table}

\subsection{Ununiform mesh with higher order finite element discretization}

In this section, we consider a more complex scenario where the underlying mesh is unstructured and the finite element discretization employs higher-order polynomial basis functions. This setting is particularly relevant for applications involving complex geometries or when enhanced accuracy is required. We generate an unstructured mesh for an L-shaped domain, defined by $[0,1]^2 \setminus [0.5,1]^2$, using Gmsh as shown in \ref{fig:lshaped}. Specifically, to accurately resolve the singularity and steep gradients characteristic of the re-entrant corner, we apply significant local mesh refinement at the inner corner $(0.5, 0.5)$, reducing the characteristic mesh size by a factor of 10 compared to the rest of the domain boundaries.
\begin{figure}[ht!]

    \centering
    \includegraphics[width=0.49\textwidth]{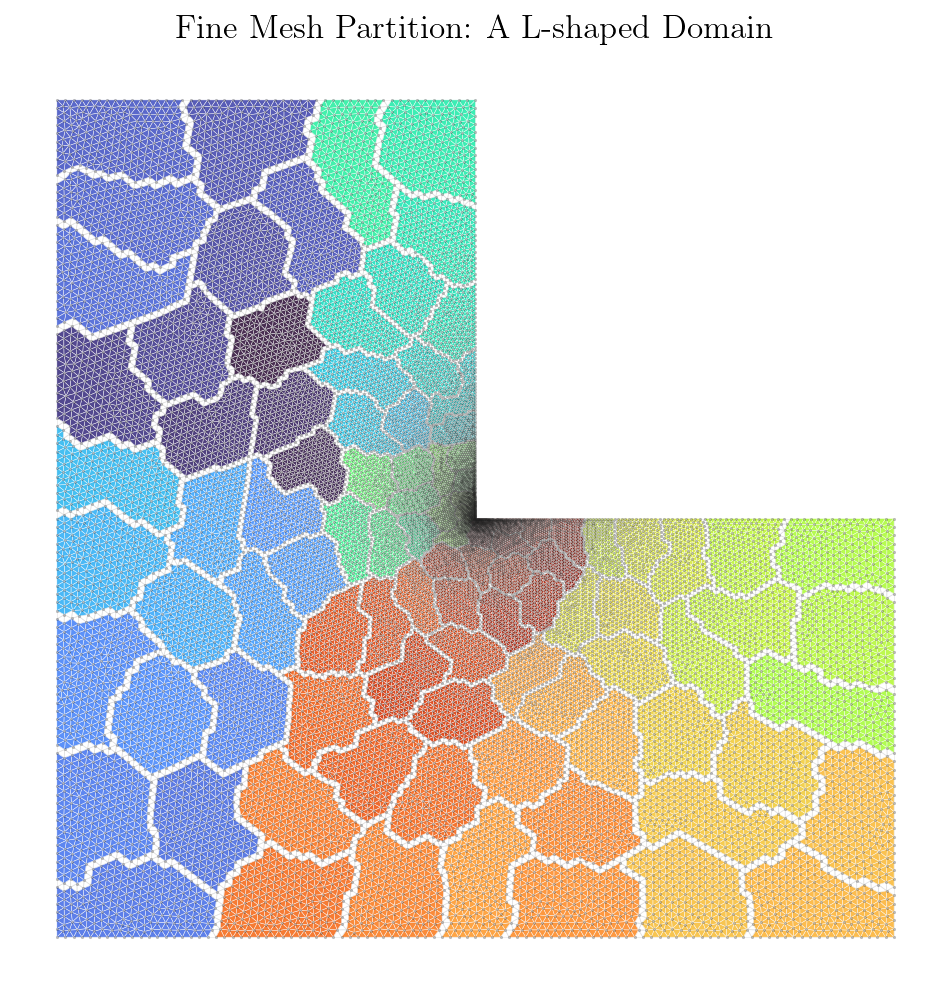}
    \caption{An unstructured mesh for an L-shaped domain, demonstrating local refinement near the reentrant corner.}
        \label{fig:lshaped}
\end{figure}

Here we fix the number of subgraphs to $n=2000$, the number of eigenvectors $Nov=4$ and vary the polynomial degree of the finite element basis functions from linear ($P_1$) to quadratic ($P_2$) and cubic ($P_3$) in the same mesh setting. The local spectral problems and energy-minimization problems are accordingly modified to accommodate the higher-order basis functions, which results in larger local systems but also richer approximation spaces.

Table \ref{tab:higher-order} presents the relative $L^2$ and energy errors for varying polynomial degrees ($P_1$, $P_2$, and $P_3$) and oversampling layers ($l$). As expected, increasing the oversampling layers $l$ leads to a rapid decay in both $L^2$ and energy errors across all polynomial degrees. For instance, with $P_2$ elements, the energy error decreases significantly from $24.92\%$ at $l=3$ to $9.59\%$ at $l=5$, and the corresponding $L^2$ error drops from $4.31\%$ to $0.72\%$. We also observe that for a fixed number of subgraphs ($n=2000$) and eigenvectors ($Nov=4$), increasing the polynomial degree from $P_1$ to $P_3$ expands the fine-scale degrees of freedom ($\#\mathcal{N}$) from 19,602 to 174,742. This results in a much richer and more detailed reference solution containing complex high-frequency spatial features, which makes the approximation slightly more challenging for a fixed macroscopic space, explaining the slight increase in relative errors. Nonetheless, the method still maintains a robust performance (e.g., around $1\%$ $L^2$ error for $l=5$ across all degrees), confirming that the CEM-GMsFEM framework effectively handles complex reference solutions discretised by higher-order finite elements on unstructured meshes.

\begin{table}[htbp]
    \centering
    \caption{Relative $L^2$ and Energy errors on the unstructured L-shaped mesh with $n=2000$ and $Nov=4$, varying polynomial degrees and oversampling layers $l$.}
    \label{tab:higher-order}
    \begin{tabular}{cccccccc}
        \toprule
        & & \multicolumn{2}{c}{$l=3$} & \multicolumn{2}{c}{$l=4$} & \multicolumn{2}{c}{$l=5$} \\
        \cmidrule(lr){3-4} \cmidrule(lr){5-6} \cmidrule(lr){7-8}
        Degree & $\#\mathcal{N}$ & $e_{L^2}^3$  & $e_{a}^3$  & $e_{L^2}^4$  & $e_{a}^4$  & $e_{L^2}^5$  & $e_{a}^5$  \\
        \midrule
        $P_1$ & 19602 & 2.69\% & 18.80\% & 1.64\% & 11.52\% & 1.35\% & 8.61\% \\
        $P_2$ & 77849 & 4.31\% & 24.92\% & 1.56\% & 14.60\% & 0.72\% & 9.59\% \\
        $P_3$ & 174742 & 5.04\% & 26.89\% & 1.92\% & 15.83\% & 1.04\% & 10.69\% \\
        \bottomrule
    \end{tabular}
\end{table}

\subsection{OpenMP acceleration performance}
\label{sec:openmp-performance}

In this subsection, we provide a reproducible benchmark framework for evaluating OpenMP acceleration without reporting separate offline and online timings. For all experiments presented here, we use the spatial network model depicted in \ref{fig:fine mesh} with a fixed high contrast $\kappa^* = 10^4$, while generating networks with different degrees of freedom. All numerical entries are intentionally left blank and will be completed after running the experiments.
We measure end-to-end wall-clock time on the workstation described at the beginning of this section (dual Intel Xeon Gold 6230, 80 cores, 1024 GB RAM). For the acceleration-effectiveness experiments in this subsection, we fix the OpenMP thread count at
\[
T=80,
\]
which is the maximum physical-thread setting available on our machine.
For each configuration, we record the two measurable runtimes
\[
t_{\mathrm{aux}}^{T},\qquad t_{\mathrm{cem}}^{T},
\]
where $t_{\mathrm{aux}}^T$ is the auxiliary-space construction time and $t_{\mathrm{cem}}^T$ is the CEM-space construction time. We define
\[
t_{\Sigma}^{T}=t_{\mathrm{aux}}^{T}+t_{\mathrm{cem}}^{T},
\]
and denote by $t_{\mathrm{aux}}^{\mathrm{noomp}}$, $t_{\mathrm{cem}}^{\mathrm{noomp}}$, and $t_{\Sigma}^{\mathrm{noomp}}$ the corresponding runtimes without OpenMP. We then compute
\[
S_{\mathrm{aux}}^{T}=\frac{t_{\mathrm{aux}}^{\mathrm{noomp}}}{t_{\mathrm{aux}}^{T}},\qquad
S_{\mathrm{cem}}^{T}=\frac{t_{\mathrm{cem}}^{\mathrm{noomp}}}{t_{\mathrm{cem}}^{T}},\qquad
S_{\Sigma}^{T}=\frac{t_{\Sigma}^{\mathrm{noomp}}}{t_{\Sigma}^{T}},
\]
\[
\eta_{\mathrm{aux}}^{T}=\frac{S_{\mathrm{aux}}^{T}}{T}\times 100\%,\qquad
\eta_{\mathrm{cem}}^{T}=\frac{S_{\mathrm{cem}}^{T}}{T}\times 100\%,
\]
where $S_{\mathrm{aux}}^{T}$ and $S_{\mathrm{cem}}^{T}$ are stage-wise speedups at fixed threads count $T$, $\eta_{\mathrm{aux}}^{T}$ and $\eta_{\mathrm{cem}}^{T}$ are stage-wise parallel efficiencies, and $S_{\Sigma}^{T}$ is the combined speedup with No OpenMP as baseline.

\subsubsection{Acceleration efficiency}
To evaluate acceleration effectiveness beyond one fixed benchmark, we design a parameterized speedup study over problem size and multiscale settings. For each configuration
\[
\omega=(\#\mathcal{N},n,l,Nov),
\]
where $\#\mathcal{N}$ is the number of vertices (degrees of freedom), we record $t_{\mathrm{aux}}$, $t_{\mathrm{cem}}$, and $t_{\Sigma}=t_{\mathrm{aux}}+t_{\mathrm{cem}}$ for both No OpenMP and OpenMP runs. 

The acceleration study is organized into two groups.
\begin{enumerate}
    \item \textbf{DoF sweep (problem-size effect).} Fix $(n,l,Nov)=(500,4,4)$ and vary
    \[
   \#\mathcal{N}\in\{11822,23608,46687,93219,290126\}.
    \]
    This group isolates how speedup changes with global degrees of freedom.
    \item \textbf{Parameter sweep (algorithmic effect).} Fix $\#\mathcal{N}=93219$ and perform one-factor-at-a-time scans:
    \[
    n\in\{500,1000,2000,4000\},\quad l\in\{2,3,4,5\},\quad Nov\in\{3,4,5,6\},
    \]
    while keeping the other two parameters at the baseline $(n,l,Nov)=(2000,4,4)$.
\end{enumerate}

For every test point, use fixed OpenMP thread count $T=80$ and one No OpenMP baseline, run one warm-up plus repeated runs, and report mean values. This protocol provides a direct answer to the two practical questions: how acceleration varies with problem size, and how sensitive speedup is to $(n,l,Nov)$ under the machine-limit thread setting.

\begin{table}[ht!]
    \centering
    \caption{Acceleration versus degrees of freedom with fixed $(n,l,Nov)=(500,4,4)$.}
    \label{tab:openmp-acc-dof}
    \begin{tabular}{cccccccc}
        \hline
         $\#\mathcal{N}$ & $n$ & $l$ & $Nov$ & $t_{\Sigma}^{\mathrm{noomp}}$ (s) & $t_{\Sigma}^{80}$ (s) & $S_{\Sigma}^{80}$ & $\eta_{\Sigma}^{80}$ (\%) \\
        \hline
        11822 & 500 & 4 & 4 & \texttt{47.88} & \texttt{2.02} & \texttt{23.70} & \texttt{29.63} \\
        23608 & 500 & 4 & 4 & \texttt{37.71} & \texttt{1.32} & \texttt{28.57} & \texttt{35.71} \\
        46687 & 500 & 4 & 4 & \texttt{104.87} & \texttt{3.54} & \texttt{29.62} & \texttt{37.03} \\
        93219 & 500 & 4 & 4 & \texttt{359.66} & \texttt{13.27} & \texttt{27.10} & \texttt{33.88} \\
        290126 & 500 & 4 & 4 & \texttt{3228.37} & \texttt{122.06} & \texttt{26.45} & \texttt{33.06} \\
        \hline
    \end{tabular}
\end{table}

\begin{table}[ht!]
    \centering
    \caption{Acceleration comparison across different $(n,l,Nov)$ at fixed $\#\mathcal{N}=93219$.}
    \label{tab:openmp-acc-params}
    \begin{tabular}{cccccccc}
        \hline
        Case  & $n$ & $l$ & $Nov$ & $t_{\Sigma}^{\mathrm{noomp}}$ (s) & $t_{\Sigma}^{80}$ (s) & $S_{\Sigma}^{80}$ & $\eta_{\Sigma}^{80}$ (\%) \\
        \midrule
        baseline  & 2000 & 4 & 4 & \texttt{155.02} & \texttt{4.35} & \texttt{35.64} & \texttt{44.55} \\
        n-scan  & 500 & 4 & 4 & \texttt{359.66} & \texttt{13.27} & \texttt{27.10} & \texttt{33.88} \\
        n-scan  & 1000 & 4 & 4 & \texttt{215.07} & \texttt{7.12} & \texttt{30.21} & \texttt{37.76} \\
        n-scan  & 4000 & 4 & 4 & \texttt{176.95} & \texttt{4.42} & \texttt{40.03} & \texttt{50.04} \\
        l-scan  & 2000 & 2 & 4 & \texttt{79.02} & \texttt{2.03} & \texttt{38.93} & \texttt{48.66} \\
        l-scan  & 2000 & 3 & 4 & \texttt{110.73} & \texttt{3.04} & \texttt{36.42} & \texttt{45.53} \\
        l-scan  & 2000 & 5 & 4 & \texttt{211.85} & \texttt{7.25} & \texttt{29.22} & \texttt{36.53} \\
        Nov-scan  & 2000 & 4 & 3 & \texttt{145.78} & \texttt{4.34} & \texttt{33.59} & \texttt{41.99} \\
        Nov-scan  & 2000 & 4 & 5 & \texttt{157.65} & \texttt{4.44} & \texttt{35.51} & \texttt{44.38} \\
        Nov-scan  & 2000 & 4 & 6 & \texttt{154.15} & \texttt{4.56} & \texttt{33.80} & \texttt{42.26} \\
        \bottomrule
    \end{tabular}
\end{table}

Table \ref{tab:openmp-acc-dof} presents the speedup and parallel efficiency trends as the global degrees of freedom ($\#\mathcal{N}$) scale from $11,822$ to $290,126$ under a fixed algorithmic configuration $(n,l,Nov)=(500,4,4)$. Remarkably, the OpenMP implementation maintains a robust and consistent acceleration across all problem dimensions, delivering combined speedups ranging from $23.70\times$ to $29.62\times$. The parallel efficiency peaks at $37.03\%$ for intermediate problem sizes ($\#\mathcal{N}=46,687$) before experiencing a marginal decline to $33.06\%$ at the largest scale ($\#\mathcal{N}=290,126$). This behavior underscores the algorithm's excellent scalability with respect to spatial complexity. Minor efficiency degradation for massive networks is an anticipated artifact of shared-memory architectures, where an exponentially growing memory footprint intensifies cache capacity contention and saturates the core-to-memory bandwidth across the $80$ hardware threads.

Table \ref{tab:openmp-acc-params} isolates the impact of multiscale algorithmic parameters---namely, the local spectral basis size $n$, oversampling layers $l$, and the neighborhood overlap $Nov$---on computational efficiency at a fixed scale ($\#\mathcal{N}=93,219$). A prominent scaling trend emerges when varying $n$: as the basis dimension $n$ increases from $500$ to $4000$, the parallel efficiency steadily climbs from $33.88\%$ to an impressive $50.04\%$, yielding a $40.03\times$ overall speedup. This phenomenon indicates that larger local eigenspaces generate denser, more compute-intensive tasks per local patch. These dense tasks effectively dominate the execution runtime, thereby perfectly amortizing the inherent thread synchronization overheads and enhancing thread availability. 

Conversely, expanding the oversampling layers $l$ from $2$ to $5$ triggers a gradual decline in parallel efficiency (from $48.66\%$ down to $36.53\%$). Because broader oversampling domains invariably increase the local spatial patch sizes processed by individual threads, the resulting heavier memory traffic leads to cache capacity misses and bandwidth bottlenecks. Finally, modulating the overlapping parameter $Nov$ produces minimal computational perturbation, with measured speedups remaining tightly bounded between $33.59\times$ and $35.51\times$. This remarkable parametric stability confirms that the workload distribution remains uniformly balanced, rendering the proposed parallel framework exceptionally resilient to structural variations within the overlapping multiscale core configurations.

\subsubsection{Strong-scaling}

Table \ref{tab:openmp-strong-template} shows a consistent strong-scaling trend for the dominant CEM stage and a saturation pattern for the lightweight auxiliary stage. For CEM, the runtime decreases monotonically from $405.25$ s (No OpenMP) to $55.64$ s ($T=10$), $32.53$ s ($T=20$), $23.45$ s ($T=40$), and $20.82$ s ($T=80$), corresponding to speedups of $7.28$, $12.46$, $17.28$, and $19.46$. This monotone reduction confirms that the main computational kernel benefits steadily from added threads throughout the tested range.
\begin{table}[ht!]
    \centering
    \caption{Strong-scaling comparison of stage-wise speedups and efficiencies (fixed $n=2000$, $l=4$, $Nov=4$, $\kappa^*=10^5$).}
    \label{tab:openmp-strong-template}
    \begin{tabular}{ccccccc}
        \hline
        Threads $T$ & $t_{\mathrm{aux}}$ (s) & $S_{\mathrm{aux}}(T)$ & $\eta_{\mathrm{aux}}(T)$ (\%) & $t_{\mathrm{cem}}$ (s) & $S_{\mathrm{cem}}(T)$ & $\eta_{\mathrm{cem}}(T)$ (\%) \\
        \midrule
        No OpenMP & \texttt{8.54} & \texttt{1.00} & -- & \texttt{405.25} & \texttt{1.00} & -- \\
        10 & \texttt{0.36} & \texttt{23.72} & \texttt{237.22} & \texttt{55.64} & \texttt{7.28} & \texttt{72.83} \\
        20 & \texttt{0.33} & \texttt{25.88} & \texttt{129.39} & \texttt{32.53} & \texttt{12.46} & \texttt{62.29} \\
        40 & \texttt{0.31} & \texttt{27.55} & \texttt{68.87} & \texttt{23.45} & \texttt{17.28} & \texttt{43.20} \\
        80 & \texttt{0.35} & \texttt{24.40} & \texttt{30.50} & \texttt{20.82} & \texttt{19.46} & \texttt{24.33} \\
        \bottomrule
    \end{tabular}
\end{table}
For the auxiliary stage, the runtime drops sharply from $8.54$ s to $0.36$--$0.31$ s for $T=10$--$40$, then slightly increases to $0.35$ s at $T=80$. The associated efficiencies above $100\%$ at low thread counts indicate apparent superlinear behavior, which is reasonable here because this stage is very short and sensitive to cache effects and runtime noise; after $T=40$, overheads (thread scheduling, synchronization, and NUMA traffic) start to dominate this small kernel. More importantly, end-to-end performance is governed by CEM: with $t_{\Sigma}=t_{\mathrm{aux}}+t_{\mathrm{cem}}$, total runtime reduces from $413.79$ s to $56.00$, $32.86$, $23.76$, and $21.17$ s as $T$ increases from No OpenMP to $10$, $20$, $40$, and $80$, respectively, yielding an overall $19.55\times$ acceleration at $T=80$. The incremental gain after each thread doubling ($7.39\times\rightarrow12.59\times\rightarrow17.42\times\rightarrow19.55\times$) also indicates diminishing returns at high concurrency, consistent with shared-memory scaling limits.

\subsubsection{Weak-scaling}
To evaluate weak scaling using the current data format, we increase the problem size and thread count proportionally. A suggested configuration is
\[
(n,\#\mathcal{N},T)\in\{(100,11822,10),(200,23608,20),(400,46687,40),(800,93219,80)\}.
\]
For each thread count $T$ (with matched problem size), we report the stage-wise runtimes $t_{\mathrm{aux}}(T)$ and $t_{\mathrm{cem}}(T)$, and the stage-wise weak-scaling efficiencies:
\[
\eta_{\mathrm{aux}}^{\mathrm{weak}}(T)=\frac{t_{\mathrm{aux}}(10)}{t_{\mathrm{aux}}(T)}\frac{\#\mathcal{N}(T)}{\#\mathcal{N}(10)}\times 100\%,\qquad
\eta_{\mathrm{cem}}^{\mathrm{weak}}(T)=\frac{t_{\mathrm{cem}}(10)}{t_{\mathrm{cem}}(T)}\frac{\#\mathcal{N}(T)}{\#\mathcal{N}(10)}\times 100\%.
\]
Table \ref{tab:openmp-weak-template} is used to report weak-scaling performance based on these four indicators.

\begin{table}[ht!]
    \centering
    \caption{Weak-scaling comparison of stage-wise runtimes and stage-wise weak-scaling efficiency.}
    \label{tab:openmp-weak-template}
    \begin{tabular}{ccccccc}
        \hline
         Vertices $\#\mathcal{N}$ & Subgraphs $n$ & Threads $T$ & $t_{\mathrm{aux}}$ (s) & $\eta_{\mathrm{aux}}^{\mathrm{weak}}(T)$ (\%) & $t_{\mathrm{cem}}$ (s) & $\eta_{\mathrm{cem}}^{\mathrm{weak}}(T)$ (\%) \\
        \midrule
         11822 & 100 & 10 & \texttt{0.14} & \texttt{100} & \texttt{3.52} & \texttt{100} \\
         23608 & 200 & 20 & \texttt{0.18} & \texttt{155.32} & \texttt{5.19} & \texttt{135.44} \\
         46687 & 400 & 40 & \texttt{0.29} & \texttt{190.65} & \texttt{6.83} & \texttt{203.53} \\
         93219 & 800 & 80 & \texttt{5.34} & \texttt{20.67} & \texttt{11.46} & \texttt{242.20} \\
        \bottomrule
    \end{tabular}
\end{table}

\begin{figure}[ht!]
    \centering
  \resizebox{\textwidth}{!}{\input{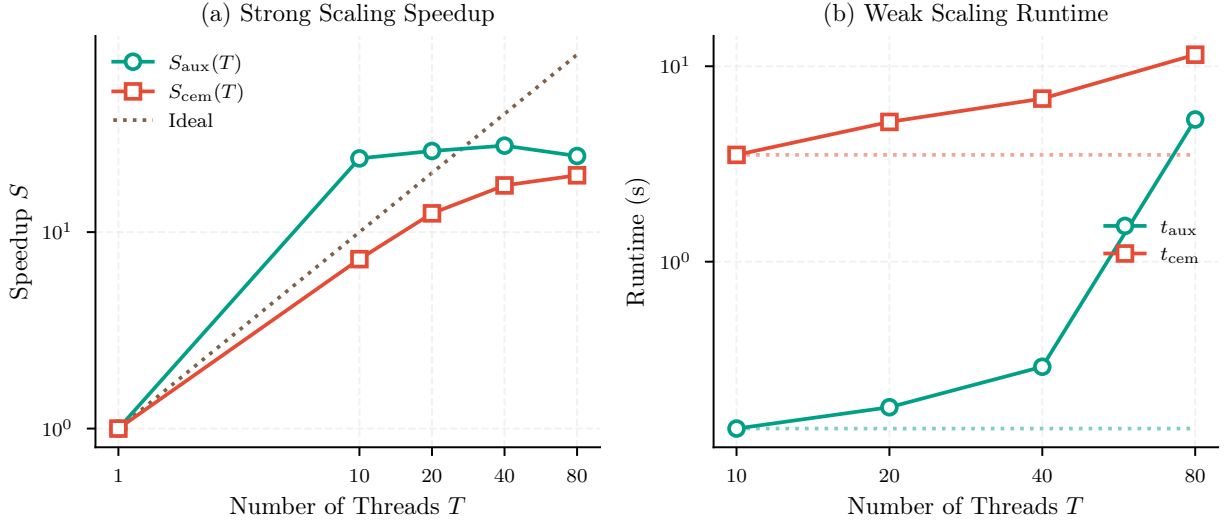}}
    \caption{Scaling performance}
    \label{fig:openmp-weak-template}
\end{figure}

In the weak-scaling tests, the problem size and thread count are increased proportionally from $(n,N_v,T)=(100,11822,10)$ to $(800,93219,80)$. Overall, the weak-scaling performance is still acceptable for $T\le 40$, since the runtimes remain relatively stable and no obvious bottleneck appears in this range. However, the performance at $T=80$ is clearly worse, mainly because the auxiliary stage time increases sharply.

Possible reasons include larger thread-management and synchronization overhead at high concurrency, stronger memory-bandwidth and NUMA penalties (more remote memory access), and task-granularity imbalance across threads. These factors can significantly reduce efficiency when the thread count reaches $80$.

\section{Conclusions}
\label{conclusions}
In this work, we introduce a geometry-parameter-independent, purely algebraic multiscale method for highly heterogeneous spatial networks. The proposed approach employs an efficient multiscale model reduction framework on spatial networks to significantly reduce computational costs. Local spectral problems are first solved on each subgraph to construct auxiliary multiscale basis functions. These are then used to generate multiscale basis functions via an energy-minimization principle applied to oversampling subgraphs. By incorporating a subgraph-wise Poincaré inequality, we define a geometry-independent constant $C_{\mathrm{po}}$, which underpins the fully algebraic nature of the method. Abstract partition of unity functions are introduced for spatial network models to facilitate the convergence analysis. Under appropriate assumptions and with suitably chosen oversampling layers, we establish an $O(C_{\mathrm{po}})$ convergence rate that is independent of the heterogeneity in edge weights and node degrees. Moreover, we show that the number of small eigenvalues is determined by the number of disconnected high-contrast inclusions or channels. A comprehensive set of numerical experiments confirms the accuracy and robustness of the proposed method.
\section*{Declaration of competing interest}

The authors declare that they have no known competing financial interests or personal relationships that could have appeared
to influence the work reported in this paper.

\section*{Declaration of Generative AI and AI-assisted technologies in the writing process}

During the preparation of this work the authors used ChatGPT in order to improve readability and language. After
using this tool, the authors reviewed and edited the content as needed and take full responsibility for the content of the
publication.

\section*{Acknowledgments}
EC's research is partially supported by the Hong Kong RGC General Research Fund (Project numbers: 14305624 and 14304525).

\bibliographystyle{siamplain}

\end{document}